\documentclass[reqno, 12pt, reqno]{amsart}

\usepackage{amsfonts, amsthm, amsmath, amssymb}
\usepackage{hyperref}

\hypersetup{colorlinks=false}

\usepackage[margin=3.5cm]{geometry}

\RequirePackage{mathrsfs} \let\mathcal\mathscr

\numberwithin{equation}{section}

\newtheorem{thm}{Theorem}[section]

\newtheorem{lem}[thm]{Lemma}

\newtheorem{prop}[thm]{Proposition}

\theoremstyle{remark}

\theoremstyle{definition}

\renewcommand{\phi}{\varphi}
\renewcommand{\rho}{\varrho}

\newcommand{\0}{\mathbf{0}}

\newcommand{\PP}{\mathbb{P}}

\newcommand{\FF}{\mathbb{F}}
\newcommand{\ZZ}{\mathbb{Z}}

\newcommand{\NN}{\mathbb{N}}
\newcommand{\QQ}{\mathbb{Q}}
\newcommand{\RR}{\mathbb{R}}
\newcommand{\CC}{\mathbb{C}}

\renewcommand{\leq}{\leqslant}

\renewcommand{\geq}{\geqslant}

\renewcommand{\bar}{\overline}

\newcommand{\h}{\mathbf{h}}

\newcommand{\x}{\mathbf{x}}
\newcommand{\y}{\mathbf{y}}

\renewcommand{\u}{\mathbf{u}}

\renewcommand{\k}{\mathbf{k}}

\newcommand{\ve}{\varepsilon}

\DeclareMathOperator{\rank}{rank}
\DeclareMathOperator{\disc}{disc}

\DeclareMathOperator{\sing}{Sing}

\begin{document}

\title[Rational points on hyperelliptic fibrations]{Uniform bounds for rational points on hyperelliptic fibrations}

\author{Dante Bonolis}
\author{Tim Browning}

\address{IST Austria\\
Am Campus 1\\
3400 Klosterneuburg\\
Austria}
\email{dante.bonolis@ist.ac.at}
\email{tdb@ist.ac.at}

\subjclass[2010]{11D45 (11L40, 11N36, 14G05)}

\begin{abstract}
We apply a variant of the square-sieve to produce an upper bound for the number of rational points of bounded height on a family of surfaces that admit a fibration over $\PP^1$  whose general fibre is  a hyperelliptic curve. 
The implied constant does not depend on the coefficients of the polynomial defining the surface.  
\end{abstract}

\date{\today}

\maketitle

\thispagestyle{empty}
\setcounter{tocdepth}{1}
\tableofcontents

\section{Introduction}

This paper is concerned with  the density of rational points on 
surfaces $S$ of the shape
\begin{equation}\label{eq:tiramisu}
Y^{2}=X^{n}+Xf(U_{1},U_{2})+g(U_{1},U_{2}),
\end{equation}
for appropriate  binary forms $f,g\in \ZZ[U_1,U_2]$, such that  $\deg(f)=2n-2$ and $\deg(g)=2n$.
We shall view $S$ as a degree $2n$ surface  in the   weighted projective space
$\PP(n,2,1,1)$, with variables $(Y,X,U_1,U_2)$. The goal of this paper is to study the  counting function 
\begin{equation}\label{eq:NB}
N(S;B)=\#\left\{(x,y,u_{1},u_{2})\in \ZZ^4: 
\begin{array}{l}
y^{2}=x^{n}+xf(u_{1},u_{2})+g(u_{1},u_{2})\\
|x|\leq B^{2},~|y|\leq B^{n},~ |u_{1}|,|u_{2}|\leq B
\end{array}
\right\},
\end{equation}
which can be interpreted in terms of counting rational points of bounded height in
$S(\QQ)$ with respect to the standard exponential height on
$\PP(n,2,1,1)(\QQ)$. Assuming that $S$ is smooth, 
the surface $S$  admits a fibration  
$S\to \PP^1$  whose general fibre is a hyperelliptic curve of  
genus $\lceil n/2-1\rceil$. 
 The following is our main result.

\begin{thm}
Let $S\subset \PP(n,2,1,1)$ be a  smooth surface given by the equation 
\eqref{eq:tiramisu}. 
Assume, furthermore, that  $n\geq 3$ is  odd and that $g$ is separable. Then 
\[
N(S;B)\ll_n B^{3-1/20}(\log B)^2,
\]
where the implied constant is only allowed to depend on $n$.
\label{thm : dp1}
\end{thm}

Consideration of the partial derivatives of the polynomial in \eqref{eq:tiramisu} shows that $S$  
is smooth precisely when there are no  roots of the  equation 
$$n^n g^{n-1}=(n-1)^{n-1} (-f)^n,$$ 
for which  $ \nabla f $ and $  \nabla g$ are proportional. 
In our work it is also necessary to assume that $g$ is separable, which we  note is equivalent to the smoothness of $S$ when $f$ is  identically zero.
The restriction on the parity of $n$ is an artefact of the proof and can be traced  to a certain exponential sum estimate
(Lemma~\ref{lem : riem}), which can fail when $n$ is even.

One way to  approach $N(S;B)$ is through work of Bombieri and Pila \cite{BP}. For any $a,b\in \ZZ$ and any choice of  $\ve>0$, 
this yields
\begin{equation}\label{eq:BP}
\#\{(x,y)\in (\ZZ\cap [-R,R])^2: y^2=x^n+a x +b\}=O_{\ve,n}(R^{1/n+\ve}),
\end{equation}
for any $R\geq 1$, 
where the implied constant is only allowed to depend on the degree $n$ and the choice of $\ve$.
An application of this with $R=B^n$ leads to the conclusion that 
\begin{equation}\label{eq:himmel}
N(S;B)=O_{\ve,n}(B^{3+\ve}).
\end{equation}
Thus our main result saves $1/20$ over this approach. 

Theorem \ref{thm : dp1} appears to be new for $n\geq 5$, but a sharper exponent is available  when $n=3$ by using better uniform bounds for counting  integer points  on elliptic curves.
Under a suitable hypothesis on the rank growth of  elliptic curves, as explained 
by Heath-Brown \cite{hb}, it is possible to conclude that the number of $x,y$ contributing to $N(S;B)$ is $O_\ve(B^\ve)$,  for any $\ve>0$, with an implied constant that only depends on $\ve$. In this way one obtains a conditional upper bound
$$
N(S;B)=O_{\ve,S}(B^{2+\ve}),
$$
when $n=3$. While we we don't yet have access to the desired conjecture on rank growth, it has recently 
been shown by Bhargava et al \cite[Thm.~1.2]{287} that 
there exists an absolute constant $c>0$ such that 
$$
\rank(E)\leq (0.2785)\log_2(|\disc(E)|)+c,
$$
for any elliptic curve $E$ in Weierstrass form with integral coefficients.
For fixed integers $|u_1|,|u_2|\leq B$,  the  elliptic curve one gets in \eqref{eq:tiramisu} has discriminant $O(B^{12})$. Once inserted into a bound of Helfgott and Venkatesh \cite[Cor.~3.9]{helf} for the number of integer points in a box that lie on an elliptic curve of given rank, this yields the estimate
$$
N(S;B)=O_S(B^{2.87}),
$$
when $n=3$. 
This is sharper than Theorem  \ref{thm : dp1} but has the defect that it depends on the coefficients of $S$,
whereas our result  is uniform in the coefficients of $f$ and $g$.
Prior to this, Mendes Da Costa \cite[\S 8]{costa} enacted a similar strategy to achieve the 
estimate
$N(S;B)=O(B^{3-\delta})$
for an  unspecified $\delta>0$, but with an absolute implied constant.
As discussed by Helfgott \cite[\S 3]{helf'},  it seems difficult to extend this strategy to any instance of the surface \eqref{eq:tiramisu} with $n>3$, since the required uniformity in the coefficients of the hyperelliptic curve is harder to come by.

The upper bound in Theorem  \ref{thm : dp1} is expected to be very far from the truth. 
We always have a lower bound $N(S;B)\gg B$ coming from solutions with $u_1=u_2=0$. 
When $n=3$ the surface $S$ is  a smooth del Pezzo surface of degree $1$ over $\QQ$ and 
Manin's conjecture
\cite{f-m-t} predicts an  upper bound of the form 
$$N(S;B)=O_S(B^2).$$
This is best possible when one of  
the 
exceptional curves that lie on $S$, 
of which there are 240 \cite[Chap.~IV]{manin}, 
 is defined over $\QQ$. However, we  actually expect linear growth outside the set of such  
 curves, by the Manin conjecture \cite{f-m-t}.

\medskip
Our proof of Theorem \ref{thm : dp1}  relies on a variant of the square sieve worked out by Pierce \cite{Pie06}, 
which allows for an application of Heath-Brown's $q$-analogue of van der Corput differencing. 
This approach was already put to use  by Heath-Brown and Pierce 
\cite{HP12} to study cyclic covers of $\PP^n$ and 
our proof  is inspired by their work. 
Ultimately, for  suitable primes $p$, 
the proof of Theorem \ref{thm : dp1} is reduced to estimating a certain $4$-variable exponential sum 
$W_p=W_p(\lambda,\h,\boldsymbol{\mu})$ defined over $\FF_p$. This sum is found in \eqref{eq:Wp}. It is fairly easy to get some cancellation in the sum, getting 
$W_p=O(p^3)$. In order to improve \eqref{eq:himmel} it is critical to get further cancellation, for 
generic  choices of parameters $\lambda,\h,\boldsymbol{\mu}$. 
While the sum is amenable to an application of work by  Katz \cite{Kat99} on singular exponential sums, this doesn't appear to yield any direct improvement. Instead, by adopting a method of moments expounded by Hooley \cite{hooleysum}, 
which we describe in the appendix, 
we can show that $W_p=O(p^{5/2})$ if 
$\lambda\neq 0$ and $(\h,\boldsymbol{\mu})\neq (\0,\0)$. It would be very interesting to gauge whether 
the sum $W_p$ actually satisfies square-root cancellation, 
for we  would then arrive at a version of Theorem \ref{thm : dp1} in which $1/20$ is replaced by $1/8$, which would be the limit of our approach.

\subsection*{Acknowledgements} 
The authors are very grateful to Roger Heath-Brown for suggesting 
the use of the  $q$-analogue of van der Corput differencing, and to Harald Helfgott, Emmanuel Kowalski, Pierre Le Boudec and Per Salberger for interesting remarks.  A great debt of thanks is also due to the 
anonymous referees, for many helpful remarks.  
While working on this paper the second  author was
supported by EPRSC 
grant \texttt{EP/P026710/1} and
FWF grant P 32428-N35.

\section{The square sieve}

\subsection{Reducing the height of the coefficients}

The implied constant in Theorem  \ref{thm : dp1} does not depend on the coefficients of $f$ or $g$. 
In fact we shall follow the convention that all of the implied constants in the remainder of our paper are only allowed to depend  on $n$, unless explicitly indicated otherwise with an appropriate subscript. 
One important step in achieving uniformity arises 
through an application of the following result, in which $\|h\|$ is used to mean the maximum of the absolute values of the coefficients of a form $h\in \ZZ[U_1,U_2]$.

\begin{lem}\label{lem:large}
Let $S\subset\mathbb{P}(n,2,1,1)$ be given by 
\eqref{eq:tiramisu}.
Then either  
$$\max\{\|f\|,\|g\|\}\ll B^{8n^{2}+4n}$$ 
or $N(S;B)=O_\ve( B^{2+\varepsilon})$ for any $\varepsilon >0$.
\end{lem}
\begin{proof}
This argument is a variant of one given by Heath-Brown \cite[Thm.~4]{annals}, but we include full details for the sake of completeness.
We shall be interested in polynomials formed from linear combinations of monomials belonging to the set 
\[
\mathcal{E}=\{Y^2\}\cup \{X^n\}\cup \{XU_{1}^{e_{1}}U_{2}^{e_{2}}: e_{1}+e_{2}=2n-2\}
\cup
\{U_{1}^{e_{1}}U_{2}^{e_{2}}: e_{1}+e_{2}=2n\}.
\]
We clearly have  $\#\mathcal{E}=4n+2$. Let 
$\mathbf{v}=(y,x,u_{1},u_{2})$ and let  $\{\mathbf{v}_{1},\dots ,\mathbf{v}_{N} \}$ be the set of all  points that are counted in $N(S;B)$. 
We construct the $N\times (4n+2)$ matrix 
$$\mathbf{C}= (\mathbf{v}_{i}^{\mathbf{e}})_{\substack{1\leq r\leq N\\\mathbf{e}\in\mathcal{E}}},
$$
 whose   $i$th row consists of the $4n+2$ possible monomials in $\mathcal{E}$ in the variables $x_{i},y_{i},u_{1,i}$, $u_{2,i}$. The matrix $\mathbf C$ has rank at most $4n+1$, since the vector $\mathbf{a}\in\mathbb{Z}^{4n+2}$ whose entries  correspond to the  coefficients of 
 \eqref{eq:tiramisu} is such that 
$
\mathbf{C} \mathbf{a}=\boldsymbol{0}.
$
We observe that  $\mathbf{a}$ is a primitive vector since its first entry is $\pm 1$.

Since $\mathbf{C}$ is not of full rank, 
the equation 
$
\mathbf{C} \mathbf{b}=\boldsymbol{0}
$
has a non-zero solution constructed from the sub-determinants of $\mathbf{C}$.
In particular, $|\mathbf{b}|=O(B^{8n^{2}+4n})$  since each entry of $\mathbf{C}$ has modulus $O(B^{2n})$.
There are two cases to consider. 
Suppose first that $\mathbf{b}$ and $\mathbf{a}$ are proportional. Then 
$|\mathbf{a}|\leq |\mathbf{b}|\ll B^{8n^2+4n}$, since $\mathbf{a}$ is a primitive vector. Alternatively, if $\mathbf{b}$ is not a multiple of $\mathbf{a}$, we let  $T\subset \PP(n,2,1,1)$ be the surface 
$B(Y,X,U_1,U_2)=0$, say, 
 corresponding to the  vector $\mathbf{b}$. Then $S\cap T$ has dimension $1$ and we claim that it has at most $4n^2$ irreducible components. 
To see this we introduce the morphism $\mathbb{P}^{3}\rightarrow\mathbb{P}(n,2,1,1)$, given by  $[z_{0},z_{1},z_{2}, z_{3}]\mapsto [z_{0}^{n},z_{1}^{2},z_{2}, z_{3}]$. 
Then the number of irreducible components of $S\cap T$ is bounded by the number of components of the intersection of 
$$
Z_0^{2n}=Z_1^{2n}+Z_1^2f(Z_2,Z_3)+g(Z_2,Z_3)
$$
with 
$B(Z_{0}^{n},Z_{1}^{2},Z_{2},Z_{3})=0$. But this has at most $4n^2$ 
irreducible components, on applying the version of B\'ezout's Theorem in  Fulton
 \cite[Example~$8.4.6$]{fulton}.

 Let $Y$ be an irreducible component of $S\cap T$. 
 Suppose there exists a
 primitive vector $(\mu_{1},\mu_{2})\in \ZZ^2$ such that $Y$ is contained in the plane 
$\mu_{2}U_{1}=\mu_{1}U_{2}$. We fix such a vector and then simply count how many vectors 
 counted by  $N(S;B)$ also satisfy $\mu_2u_1=\mu_1u_2$. 
Assume that $\mu_1\neq 0$. Then this quantity 
is bounded by the number of vectors $(y,x,u_1)\in \ZZ^3$  with 
$$
\mu_1^{2n}y^2=\mu_1^{2n}x^n+\mu_1^2x f(\mu_1,\mu_2)u_1^{2n-2}+
g(\mu_1,\mu_2)u_1^{2n}.
$$
For each $u_1$,  we may appeal to the Bombieri--Pila bound \eqref{eq:BP} to get $O_\ve(B^{1+\ve})$ possibilities for $x,y$, for any $\ve>0$. This case therefore gives an overall 
contribution of $O_\ve(B^{2+\ve})$. Next, we may suppose that 
$Y\not\subset P$ for every 
plane $P\subset \PP(n,2,1,1)$ with equation
$
\mu_{2}U_{1}=\mu_{1}U_{2},
$ 
as $(\mu_{1},\mu_{2})\in \ZZ^2$
runs over  primitive vectors. In particular $\#(Y\cap P)=O(1)$  by 
B\'ezout's theorem. Since any non-zero vector $(u_1,u_2)\in \ZZ^2$ with $|u_1|,|u_2|\leq B$ 
satisfies the equation defining $P$ for at least one primitive vector $(\mu_{1},\mu_{2})\in \ZZ^2$ with norm at most  $B$, we easily obtain a contribution of $O(B^2)$ in this case.
This completes the proof of the lemma.
\end{proof}

The surface $S$ has a discriminant $D_S$ that is an integer polynomial 
in the coefficients of $f$ and $g$, and which vanishes precisely when $S$ is singular.  $D_S$ can be calculated using elimination theory, by following the arguments in 
\cite{GKZ}.  Let $\Delta_{f,g}$ be the absolute value of the product of $D_S$ and the discriminant of the binary form $g$.
Then 
$\Delta_{f,g}$ is an integer, which is a polynomial in the coefficients of $f$ and $g$, and which vanishes precisely when the surface $S$ is singular, or when $g$ has a repeated root. 
 Taken together with our hypotheses in Theorem \ref{thm : dp1},  Lemma \ref{lem:large} 
allows us to proceed under the assumption that
$\Delta_{f,g}$
is a positive  integer such that 
\begin{equation}\label{eq:Delta}
\log \Delta_{f,g} =O(\log B),
\end{equation}
where we recall our convention that the implied constant in any estimate is allowed to depend on $n$.
For  any prime $p$, if $p\nmid \Delta_{f,g}$ then 
 the reduction of $S$ modulo $p$ is smooth and the reduction modulo $p$ of $g$ has no repeated roots.

\subsection{Application of the sieve}

We shall prove Theorem \ref{thm : dp1} using the variant of Heath-Brown's square sieve introduced in \cite{Pie06}.  This sieve offers a great deal of flexibility in the sieving set of primes, which we shall take advantage of here. 
We shall estimate  $N(S;B)$ by  sieving for squares in the non-negative sequence
\[
\omega (m)=\#\left\{(x,u_{1},u_{2})\in \mathbb{Z}^{3}:
\begin{array}{l}
m=x^{n}+xf(u_{1},u_{2})+g(u_{1},u_{2})\\
|x|\leq B^{2}, ~|u_{1}|,|u_{2}|\leq B
\end{array}
\right\}.
\]
Let $P,Q\geq 1$ be parameters depending on $B$ which are to be determined   in due course. For now we shall merely assume that 
\begin{equation}\label{eq:PQ}
Q\leq \sqrt{B}\leq P\leq B, \quad PQ\geq B.
\end{equation}

It is now time to reveal the sieve. 
Let 
\begin{align*}
\mathcal{P}&=\{p \text{ prime}: p\equiv 2\bmod{n} \text{, } p\nmid  \Delta_{f,g} \text{ and }P\leq p\leq 2P\},\\ 
\mathcal{Q}&=\{q \text{ prime}: Q\leq q\leq 2Q\}
\end{align*}
and 
$$
\mathcal{A}=\{p\cdot q : (p,q)\in\mathcal{P}\times\mathcal{Q}\}.
$$
On assuming $B\geq 4n^2$ we clearly have  $p\nmid 2n$ for any $p\in \mathcal{P}.$
According to Pierce \cite[Lemma 2.1]{Pie06},  we have 
\begin{align*}
\sum_{m}\omega (m^{2})&\ll \frac{1}{\#\mathcal{A}}\sum_{m}\omega (m)+
\frac{1}{\#\mathcal{A}^{2}}\sum_{\substack{p,p'\in\mathcal{P}\\p\neq p'}}\sum_{\substack{q,q'\in \mathcal{Q}\\q\neq q'}}\left|\sum_{m}\omega (m)\left(\frac{m}{pq}\right)
\left(\frac{m}{p'q'}\right)\right|\\&\quad +
\frac{\#\mathcal{Q}}{\#\mathcal{A}^{2}}\sum_{\substack{p,p'\in\mathcal{P}\\p\neq p'}}\left|\sum_{m}\omega (m)\left(\frac{m}{pp'}\right)\right|+\frac{1}{\#\mathcal{A}^{2}}|E(\mathcal{P})|\\&\quad + \frac{\#\mathcal{P}}{\#\mathcal{A}^{2}}\sum_{\substack{q,q'\in \mathcal{Q}\\q\neq q'}}\left|\sum_{m}\omega (m)\left(\frac{m}{qq'}\right)\right|+\frac{1}{\#\mathcal{A}^{2}}|E(\mathcal{Q})|,
\end{align*}
where 
\begin{equation}\label{eq:EE}
E(\mathcal{P})=\sum_{q\in\mathcal{Q}}\sum_{\substack{p,p'\in \mathcal{P}\\ p\neq p'}}\sum_{\substack{m\\  q\mid m}}\omega (m)\left(\frac{m}{pp'}\right), \quad
E(\mathcal{Q})=\sum_{p\in\mathcal{P}}\sum_{\substack{q,q'\in \mathcal{Q}\\q\neq q'}}\sum_{\substack{m\\  p\mid m}}\omega (m)\left(\frac{m}{qq'}\right).
\end{equation}

For the first sum on the right hand side we trivially have  $\sum_{m}\omega (m)\ll B^{4}$. 
Next, we clearly have
\begin{align*}
\#\mathcal{P}
&\geq \#\{p \text{ prime}: p\equiv 2\bmod{n}  \text{ and }P\leq p\leq 2P\} - 
\omega( \Delta_{f,g})\\
&\gg \frac{P}{\log B},
\end{align*}
since 
$P\leq B^2$ and 
\eqref{eq:Delta} ensures that 
$
\omega( \Delta_{f,g})\leq \log \Delta_{f,g} \ll \log B.
$
Moreover, 
$\#\mathcal{Q}\gg Q/\log B$, since $Q\leq \sqrt{B}$.
Hence, in view of \eqref{eq:NB},  we have 
\begin{equation}
\label{eq : sieve}
\begin{split}
N(S;B)&\ll  \frac{B^4(\log B)^2}{PQ}
+\frac{1}{\#\mathcal{A}^{2}}
\sum_{\substack{p,p'\in\mathcal{P}\\p\neq p'}}\sum_{\substack{q,q'\in \mathcal{Q}\\q\neq q'}}\left|
C(pp'qq')\right|\\
&\quad +\frac{\log B}{Q\#\mathcal{P}^{2}}
\sum_{\substack{p,p'\in\mathcal{P}\\p\neq p'}}\left|C(pp')\right|
+\frac{1}{\#\mathcal{A}^{2}}|E(\mathcal{P})|\\&\quad +
\frac{\log B}{P\#\mathcal{Q}^{2}}
\sum_{\substack{q,q'\in \mathcal{Q}\\q\neq q'}}\left| C(qq')\right|+\frac{1}{\#\mathcal{A}^{2}}|E(\mathcal{Q})|,
\end{split}
\end{equation}
where
\begin{equation}
\begin{split}
C(r)
&=\sum_{m}\omega (m)\left(\frac{m}{r}\right)\\
&=
\sum_{\substack{|u_{1}|,|u_{2}|\leq B\\|x|\leq B^{2}}}\left(\frac{x^{n}+xf(u_{1},u_{2})+g(u_{1},u_{2})}{r}\right),
\label{eq : main}
\end{split}
\end{equation}
for any  square-free $r\in \NN$.

\subsection{The main oscillatory sum}

This section is devoted to  bounding the sum $C(r)$, as defined in \eqref{eq : main}, for various choices of square-free $r\in \NN$. 
We have 
\begin{align*}
C(r)=~&\sum_{|u_{1}|,|u_{2}|\leq B}\sum_{\alpha \bmod{r}}\left(\frac{\alpha^{n}+\alpha f(u_{1},u_{2})+g(u_{1},u_{2})}{r}\right)\\
&\times \sum_{|x|\leq B^{2}}\frac{1}{r}\sum_{c=1}^{r}e_{r}(c(\alpha-x)),
\end{align*}
on using additive characters to detect the congruence. 
Define 
\begin{equation}\label{eq:S}
S(r,c,u_{1},u_{2})=\sum_{\alpha \bmod{r}}\left(\frac{\alpha^{n}+\alpha f(u_{1},u_{2})+g(u_{1},u_{2})}{r}\right)e_{r}(c\alpha)
\end{equation}
and
$$
U(r,c,B)=\sum_{|u_{1}|,|u_{2}|\leq B}S(r,c,u_{1},u_{2}).
$$
Then we deduce that 
\begin{equation}\label{eq:rosso}
C(r)\ll 
\frac{1}{r}\sum_{c=1}^{r}\min \left(B^{2},\left\|\frac{c}{r}\right\|^{-1}\right)|U(r,c,B)|.
\end{equation}
The exponential sum $S(r,c,u_1,u_2)$ in \eqref{eq:S} satisfies the following basic multiplicativity property.

\begin{lem}
Assume that $r=r_{0}r_{1}$ with $\gcd(r_{0},r_{1})=1$. Then
\[
S(r,c,u_{1},u_{2})=S(r_{0},c\overline{r}_{1},u_{1},u_{2})S(r_{1},c\overline{r}_{0},u_{1},u_{2}),
\]
where $r_{1}\overline{r}_{1}\equiv 1\bmod{r_{0}}$ and $r_{0}\overline{r}_{0}\equiv 1\bmod{r_{1}}$.
\label{lem : splitprime}
\end{lem}
\begin{proof}
This result follows easily on using the Chinese remainder theorem to note that 
$\alpha_1r_0+\alpha_0r_1$ runs through all residue classes modulo $r$ as 
$\alpha_0$ runs through residue classes modulo $r_0$ and 
$\alpha_1$ runs through residue classes modulo $r_1$.
Thus \begin{align*}
S(r,c,u_{1},u_{2})
&=\left(\sum_{\alpha_{0} =1}^{r_{0}}\left(\frac{(\alpha_{0}r_{1})^{n}+(\alpha_{0}r_{1}) f(u_{1},u_{2})+g(u_{1},u_{2})}{r_{0}}\right)e_{r_{0}}(c\alpha_{0})\right)\\&
\quad\times\left(\sum_{\alpha_{1} =1}^{r_{1}}\left(\frac{(\alpha_{1}r_{0})^{n}+(\alpha_{1}r_{0}) f(u_{1},u_{2})+g(u_{1},u_{2})}{r_{1}}\right)e_{r_{1}}(c\alpha_{1})\right)\\& =S(r_{0},c\overline{r}_{1},u_{1},u_{2})S(r_{1},c\overline{r}_{0},u_{1},u_{2}),
\end{align*}
on making a change of variables.
\end{proof}

We shall need to 
bound the exponential sum 
$S(r,c,u_1,u_2)$ in \eqref{eq:S}, for given $\u\in \ZZ^2$. 
Lemma \ref{lem : splitprime} ensures that it suffices to look at prime values of $r$, in which setting 
the following result demonstrates that square-root cancellation occurs.

\begin{lem}
Let $p$ be a prime and let $n$ be odd.
For any $a,b,c\in\mathbb{F}_{p}$ there is a constant $C_n>0$ such that 
$$
\left|\sum_{x\in\mathbb{F}_{p}}\left(\frac{x^{n}+ax+b}{p}\right)e_{p}(cx)\right|\leq C_n\sqrt{p}.
$$
\label{lem : riem}
\end{lem}

\begin{proof}
On adjusting $C_n$ we can assume that  $p$  is an odd prime. We start by observing that if $\theta\in\overline{\FF}_{p}$ is a root $T^{n}+aT+b$ of multiplicity $r$, then 
$$
T^{n}+aT+b=\Phi_{\theta} (T)^{r}\Psi(T),
$$ 
where $\Phi_{\theta}(T)$ is the minimal polynomial of $\theta$ over $\FF_{p}$ and $\Psi(T)\in\FF_{p} [T]$ is such that $\Psi(\theta)\neq 0$. 
We claim that  there exist  polynomials $h_{1},h_{2}\in\FF_{p}[T]$ such that $T^{n}+aT+b=h_{1}h_{2}^{2}$. If $r$ is even this is obvious with 
$h_1= \Psi$ and 
$h_2=\Phi_\theta^{r/2}$. If $r=2k+1$ is odd, then we take 
$h_1=\Psi\Phi_\theta$ and $h_2=\Phi_\theta^{k}$.
It is clear that  $h_{1}$ is non-constant,
since we are assuming $n$ to be  odd. 
Moreover, we can assume that  $h_{1}$ separable, since any square factors can be absorbed into the term $h_2^2$. 
For any $x$ such that $h_{2}(x)\neq 0$, we have 
\[
\left(\frac{x^{n}+ax+b}{p}\right)=\left(\frac{h_{1}(x)h_{2}(x)^{2}}{p}\right)=\left(\frac{h_{1}(x)}{p}\right).
\]
Thus
\begin{align*}
\sum_{x\in\mathbb{F}_{p}}\left(\frac{x^{n}+ax+b}{p}\right)e_{p}(cx)
&= 
\sum_{\substack{x\in\mathbb{F}_{p}\\ h_{2}(x)\neq 0}}
\left(\frac{h_{1}(x)}{p}\right)e_{p}(cx)\\
&=
\sum_{x\in\mathbb{F}_{p}}\left(\frac{h_{1}(x)}{p}\right)e_{p}(cx) +O(1)\\
&\leq C_n\sqrt{p},
\end{align*}
thanks to Theorems 2B and  2G in Schmidt \cite[Chapter II]{Sch76}.
\end{proof}

Note that if $n$ were even the left hand side would be $p$ if $a=b=c=0$.  
Thus it is crucial to assume that $n$ is odd in  
Lemma \ref{lem : riem}, in order to have a result that applies to all $a,b,c\in \FF_q$. 
 We are now ready to record our first result for $U(r,c,B)$.

\begin{lem}\label{lem:bianco1}
Assume that $n$ is odd and let   $r\in \NN$ be  square-free. 
There exists a constant $C_n>0$ depending only on $n$ such that 
$$
U(r,c,B)\leq C_n^{\omega(r)} 
B^2 r^{1/2}.
$$
\end{lem}

\begin{proof}
This is an easy consequence of   Lemmas 
\ref{lem : splitprime} and 
 \ref{lem : riem}.
\end{proof}

The previous estimate will be enough to handle all but the second term in \eqref{eq : sieve}.  To handle the case $r=pp'qq'$ 
for  distinct primes  $p,p'\in\mathcal{P}$ and $q,q'\in\mathcal{Q}$, 
it will be convenient to set 
$r_{0}=pp'$ and  $r_{1}=qq'$. 
 We observe that $r_0\asymp P^2$ and 
 $r_{1}\asymp Q^{2}$. Since $PQ\leq B^{3/2}$ in \eqref{eq:PQ} we deduce that  the range of summation for $x$ is $B^{2}\gg \sqrt{r_{0}r_{1}}$. Hence it makes sense to complete the summation over $x$ to all the classes modulo $r_{0}r_{1}$, as we have done here.
The following estimate for  $U(r_{0}r_{1},c,B)$ is obtained using the $q$-analogue of the van der Corput 
inequality. 

\begin{lem}\label{lem:bianco2}
We have
$$
U(r_0r_1,c,B)\ll 
\begin{cases}
Br_0^{1/2}r_1^{3/2}+
B r_0^{5/4}r_1^{1/2}\log r_0 &\text{ if $\gcd(c,r_{0})=1$,}\\ 
B^2 (r_0r_1)^{1/2} &
\text{ if $\gcd(c,r_{0})>1$.}\end{cases}
$$
\end{lem}

The proof of this result will occupy the remainder of this subsection. 
We start the proof by defining 
\[
A(u_{1},u_{2})=
\begin{cases}
S(r,c,u_{1},u_{2})		&\text{ if $|u_{1}|,|u_{2}|\leq B$,}\\
0						&\text{ otherwise,}
\end{cases}
\]
and
\[
A_{0}(u_{1},u_{2})=
\begin{cases}
S(r_{0},\overline{r}_{1}c,u_{1},u_{2})		&\text{ if $|u_{1}|,|u_{2}|\leq B$,}\\
0						&\text{otherwise.}
\end{cases}
\]
We define  $A_{1}(u_{1},u_{2})$  similarly and we introduce the parameter
$$
H=\left[\frac{4B}{r_{1}}\right].
$$ 
Since $r_1=qq'\leq 4Q^2\leq 4B$ by \eqref{eq:PQ}, we see that $H\in \NN$. 
We will follow the $q$-analogue of the van der Corput method. First of all,  
we find that  
\begin{align*}
H^{2}U(r,c,B)&=\sum_{\mathbf{h}\in[1, H]^{2}}\sum_{\mathbf{u}\in \ZZ^2}
A(\mathbf{u}+\mathbf{h}r_{1})\\
&=\sum_{\mathbf{u}\in\mathbb{Z}^{2}}\sum_{\mathbf{h}\in[1, H]^{2}}A_{0}(\mathbf{u}+\mathbf{h}r_{1})A_{1}(\mathbf{u}+\mathbf{h}r_{1})\\
&=\sum_{\mathbf{u}\in\mathbb{Z}^2}S(r_{1},\overline{r}_{0}c,u_{1},u_{2})\sum_{\mathbf{h}\in[1, H]^{2}}A_{0}(\mathbf{u}+\mathbf{h}r_{1}),
\end{align*}
where $\mathbf{u}=(u_{1},u_{2})$ and $\mathbf{h}=(h_{1},h_{2})$. Let 
$|\cdot|$ be the sup norm on $\RR^2$.
By the Cauchy--Schwarz inequality we get 
\begin{equation}
H^{2}|U(r,c,B)|\leq\sqrt{\Sigma_{1}\Sigma_{2}},
\label{eq : vancorsigma}
\end{equation}
where
\begin{equation}\label{eq:sigma1}
\Sigma_{1}=
\max_{\h \in [1,H]^2} 
\sum_{
\substack{\mathbf{u}\in\mathbb{Z}^{2}\\
|\u+\h r_1|\leq B}}
|S(r_{1},c\overline{r}_{0},u_{1},u_{2})|^{2}
\end{equation}
and
\[
\Sigma_{2}=\sum_{\mathbf{u}\in\mathbb{Z}^{2}}\left|\sum_{\mathbf{h}\in[1, H]^{2}}A_{0}(\mathbf{u}+\mathbf{h}r_{1})\right|^{2}.
\]
Moreover,
\[
\begin{split}
\Sigma_{2}&=\sum_{\mathbf{h}\in[1, H]^{2}}
\sum_{\mathbf{j} \in [1, H]^{2}}\sum_{\mathbf{u}\in\mathbb{Z}^{2}}A_{0}(\mathbf{u}+\mathbf{h}r_{1})\overline{A_{0}(\mathbf{u}+\mathbf{j}r_{1})}\\
&=\sum_{\mathbf{h}\in[1, H]^{2}}\sum_{\mathbf{j}\in [1,H]^{2}}\sum_{\mathbf{u}\in\mathbb{Z}^{2}}A_{0}(\mathbf{u}+(\mathbf{h}-\mathbf{j})r_{1})\overline{A_{0}(\mathbf{u})}\\&\leq 2H^{2}\sum_{
\substack{\mathbf{h}\in \ZZ^2\\ 
|\h|\leq H}}
\left|\sum_{\mathbf{u}\in\mathbb{Z}^{2}}A_{0}(\mathbf{u}+\mathbf{h}r_{1})\overline{A_{0}(\mathbf{u})}\right|.
\end{split}
\]
We have 
\begin{equation}\label{eq:sigma2}
\Sigma_2\leq 2H^2(\Sigma_{2,A}+\Sigma_{2,B}),
\end{equation}
where
\begin{equation}\label{eq:sigma2A}
\Sigma_{2,A}=\sum_{\mathbf{u}\in\mathbb{Z}^{2}}|A_{0}(\mathbf{u})|^{2}
\end{equation}
and
\begin{equation}\label{eq:sigma2B}
\Sigma_{2,B}=\sum_{\substack{\mathbf{h}\in \ZZ^2\\
0<|\h|\leq H}}
\left|\sum_{\mathbf{u}\in\mathbb{Z}^{2}}A_{0}(\mathbf{u}+\mathbf{h}r_{1})\overline{A_{0}(\mathbf{u})}\right|.
\end{equation}

The following  estimate is enough 
to complete the treatment of $\Sigma_1$ and $\Sigma_{2,A}$
in 
\eqref{eq:sigma1} and 
\eqref{eq:sigma2A}, respectively.

\begin{lem}\label{lem:pizza}
We have $\Sigma_1=O(B^2r_1)$ and 
$\Sigma_{2,A}=O(B^2r_0).$
\end{lem}

\begin{proof}
Appealing to Lemma \ref{lem : riem} and the multiplicativity property in Lemma \ref{lem : splitprime}, we deduce that 
$$
\Sigma_1\ll (B+Hr_1)^2r_1 \quad \text{ and } \quad
\Sigma_{2,A}\ll B^2r_0.
$$
The lemma follows on noting that  $Hr_1=[B/r_1]r_1\leq B$.
\end{proof}

We now turn to the estimation of $\Sigma_{2,B}$, as defined in \eqref{eq:sigma2B}. We may  write
\begin{equation}\label{eq:pasta}
\Sigma_{2,B}=\sum_{\substack{\mathbf{h}\in \ZZ^2\\ 0<|\h|\leq H}}
|T(r_{0},\mathbf{h})|,
\end{equation}
where
\[
\begin{split}
T(r_{0},\mathbf{h})
&=\sum_{-B\leq u_{1}\leq B-h_{1}r_{1}}\sum_{-B\leq u_{2} \leq B-h_{2}r_{1}}S(r_{0},c\overline{r}_{1},\mathbf{u}+\mathbf{h}r_{1})
\overline{S(r_{0},c\overline{r}_{1},\mathbf{u})}\\
&=\sum_{s_{1},s_{2}\bmod{r_0}}S(r_{0},c\overline{r}_{1},\mathbf{s}+\mathbf{h}r_{1})
\overline{S(r_{0},c\overline{r}_{1},\mathbf{s})}\\
&\quad \times
\left(\sum_{-B\leq u_{1} \leq B-h_{1}r_{1}}\frac{1}{r_{0}}\sum_{k_{1}=1}^{r_{0}}e_{r_{0}}(k_{1}(s_{1}-u_{1}))\right)\\
&\quad\times \left(\sum_{-B\leq u_{2} \leq B-h_{2}r_{1}}\frac{1}{r_{0}}\sum_{k_{2}=1}^{r_{0}}e_{r_{0}}(k_{2}(s_{2}-u_{2}))\right).
\end{split}
\]
It follows that 
\[
T(r_{0},\mathbf{h})\leq\frac{1}{r_{0}^{2}}\sum_{k_{1},k_{2}=1}^{r_{0}}\min \left(B,\left\|\frac{k_{1}}{r_{0}}\right\|^{-1}\right)\min \left(B,\left\|\frac{k_{2}}{r_{0}}\right\|^{-1}\right) |W(\k)|,
\]
where $\k=(k_1,k_2)$ and  
\begin{align*}
W(\k)=\sum_{\alpha,\beta,s_{1},s_{2}\bmod r_0}&\left(\frac{\alpha^{n}+\alpha f(\mathbf{s}+\mathbf{h}r_1)+ g(\mathbf{s}+\mathbf{h}r_1)}{r_0}\right)\\
&\times \left(\frac{\beta^{n}+\beta f(\mathbf{s})+ g(\mathbf{s})}{r_0}\right)e_{r_0}(c\bar r_1(\alpha-\beta)+\k\cdot\mathbf{s}).
\end{align*}
Define the exponential sum
\begin{equation}\label{eq:Wp}
\begin{split}
W_p(\lambda,\mathbf{h},\boldsymbol{\mu})=
\hspace{-0.3cm}
\sum_{\alpha,\beta,s_{1},s_{2}\bmod p}&\left(\frac{\alpha^{n}+\alpha f(\mathbf{s}+\mathbf{h})+ g(\mathbf{s}+\mathbf{h})}{p}\right)\\
& \times \left(\frac{\beta^{n}+\beta f(\mathbf{s})+ g(\mathbf{s})}{p}\right)e_{p}(\lambda(\alpha-\beta)+\boldsymbol{\mu}\cdot\mathbf{s}),
\end{split}
\end{equation}
for $\lambda\in \ZZ$ and $\h,\boldsymbol{\mu}\in \ZZ^2$. 
It now follows from the  Chinese remainder theorem that 
$$
W(\k)=
W_p(c\overline{r}_{1}\overline{p'},\mathbf{h}r_{1},\boldsymbol{k}\overline{p'})W_{p'}(c\overline{r}_{1}\overline{p},\mathbf{h}r_{1},\boldsymbol{k}\overline{p}),
$$ 
since $r_0=pp'$.

Thus our attention shifts to estimating 
$W_p(\lambda,\mathbf{h},\boldsymbol{\mu})$. The trivial bound is $O(p^4)$. 
The bound $O(p^3)$ follows rather easily from Lemma \ref{lem : riem}. Any non-trivial saving over this bound  will yield an improvement over the  bound 
\eqref{eq:himmel}. 
Unfortunately we are not able to achieve full square-root cancellation for 
$W_p(\lambda,\mathbf{h},\boldsymbol{\mu})$. The following result summarises our analysis and will be established in Section \ref{s:Wp}.

\begin{prop}
Let $p\equiv 2\bmod{n}$ be a prime such that 
$p\nmid \Delta_{f,g}$.  Let  $\lambda\in\mathbb{F}_{p}^{\times}$.  Then 
\[
W_p(\lambda,\mathbf{h},\boldsymbol{\mu})\ll p^{5/2}\gcd(p,h_{1},h_{2},\mu_{1},\mu_{2})^{1/2},
\]
where the implied constant depends at most on $n$. 
\label{prop : expsum}
\end{prop}

We are now ready to produce our final estimate for $\Sigma_{2,B}$, as defined in 
 \eqref{eq:pasta}. 

\begin{lem}\label{lem:gnocchi}
Assume that $\gcd(c,r_0)=1$. Then  $\Sigma_{2,B}=O(H^{2}r_0^{5/2}(\log r_0)^2)$.
\end{lem}

\begin{proof}
The assumption $\gcd(c,r_0)=1$ brings us in line for an application of Proposition \ref{prop : expsum}, 
since $p, p'\equiv 2 \bmod{n}$ for any $p,p'\in \mathcal{P}$.
Hence 
\[
T(r_{0},\mathbf{h})
\hspace{-0.1cm}\ll \hspace{-0.1cm}
r_{0}^{1/2}
\hspace{-0.2cm}
\sum_{k_{1},k_{2}=1}^{r_{0}}\min \left(B,\left\|\frac{k_{1}}{r_{0}}\right\|^{-1}\right)\min \left(B,\left\|\frac{k_{2}}{r_{0}}\right\|^{-1}\right)\gcd(r_{0},h_{1},h_{2},k_{1},k_{2})^{1/2}.
\]
Inserting this into  \eqref{eq:pasta}, we obtain
\[
\begin{split}
\Sigma_{2,B}&\ll 
B^{2}r_{0}^{1/2}
\sum_{\substack{\mathbf{h}\in \ZZ^2\\ 0<|\h|\leq H}}
\gcd(r_{0},h_{1},h_{2})^{1/2}
\\
&\qquad+ Br_{0}^{1/2}
\sum_{\substack{\mathbf{h}\in \ZZ^2\\ 0<|\h|\leq H}}
\sum_{k=1}^{r_{0}-1}\left\|\frac{k}{r_{0}}\right\|^{-1} \gcd(r_{0},h_{1},h_{2},k)^{1/2}\\
&\qquad +r_{0}^{1/2}
\sum_{\substack{\mathbf{h}\in \ZZ^2\\ 0<|\h|\leq H}}
\sum_{k_{1},k_{2}=1}^{r_{0}-1}\left\|\frac{k_{1}}{r_{0}}\right\|^{-1}
 \left\|\frac{k_{2}}{r_{0}}\right\|^{-1}\gcd(r_{0},h_{1},h_{2},k_{1},k_{2})^{1/2}.
\end{split}
\]
The third term is plainly
\begin{align*}
&\ll r_0^{1/2}\sum_{k_{1},k_{2}=1}^{r_{0}/2}\frac{r_{0}^{2}}{k_{1}k_{2}}
\sum_{\substack{\mathbf{h}\in \ZZ^2\\ 0<|\h|\leq H}}
\gcd(r_{0},h_{1},h_{2},k_{1},k_{2})^{1/2}\\
& \ll \sum_{k_{1},k_{2}=1}^{r_{0}/2}\frac{r_{0}^{5/2}}{k_{1}k_{2}}\sum_{d\mid \gcd(r_{0},k_{1},k_{2})}d^{1/2}\#\{\mathbf{h}\in \ZZ^2: 0<|\h|\leq H \text{ and } d\mid \h\}
\\
&\ll H^2r_{0}^{5/2}(\log r_0)^2.
\end{align*}
Using a similar argument for the remaining two terms, we deduce that 
\begin{align*}
\Sigma_{2,B}
&\ll H^{2}B^{2}r_{0}^{1/2}+
H^{2}B r_{0}^{3/2}\log r_0
+ H^{2}r_{0}^{5/2}(\log r_0)^2.
\end{align*}
The lemma follows since 
$r_{0}\asymp P^2\geq B$, by
\eqref{eq:PQ}.
\end{proof}

We now have everything in place to 
estimate $U(r_0r_1,c,B)$ and so 
complete 
 the proof of Lemma \ref{lem:bianco2}. 
If $\gcd(c,r_{0})>1$ we merely 
apply  Lemma \ref{lem:bianco1}.
On the other hand, 
if $\gcd(c,r_{0})=1$ we return to \eqref{eq : vancorsigma} and 
\eqref{eq:sigma2}, in order to deduce that 
\[
U(r,c,B)\ll H^{-1}\Sigma_{1}^{1/2}(\Sigma_{2,A}+\Sigma_{2,B})^{1/2}.
\]
Inserting the bounds for $\Sigma_{1},\Sigma_{2,A}$ and $\Sigma_{2,B}$
from Lemmas \ref{lem:pizza} and
\ref{lem:gnocchi}, 
\begin{align*}
U(r,c,B)
&\ll H^{-1} \cdot Br_1^{1/2} \cdot \left(Br_0^{1/2} +
H r_0^{5/4}\log r_0\right)\\
&\ll  \frac{B^2(r_0r_1)^{1/2}}{H} +
B r_0^{5/4}r_1^{1/2}\log r_0.
\end{align*}
This therefore completes the proof of Lemma~\ref{lem:bianco2}, 
since   $H=[B/r_1]\gg B/r_1$.

\subsection{Completion of the proof of Theorem \ref{thm : dp1}}

It is now time to return to the upper bound for $N(S;B)$ in \eqref{eq : sieve}. 
The following lemmas are  devoted to dealing with the various terms that appear in this expression. 

\begin{lem}\label{lem:1}
Assume that $P,Q$ satisfy \eqref{eq:PQ}. Then 
$$
\frac{1}{\#\mathcal{A}^{2}}
\sum_{\substack{p,p'\in\mathcal{P}\\p\neq p'}}\sum_{\substack{q,q'\in \mathcal{Q}\\q\neq q'}}\left|
C(pp'qq')\right|
\ll 
 \left(BPQ^3+ BP^{5/2}Q+\frac{B^4}{PQ}\right)(\log B)^2.
$$
\end{lem}

\begin{proof}
Applying Lemma~\ref{lem:bianco2} in \eqref{eq:rosso}, we obtain
\begin{align*}
C(r_{0}r_{1})\ll~& 
\frac{1}{r_{0}r_{1}}\sum_{\substack{c=1\\ 
\gcd(c,r_0)=1}
}^{r_{0}r_{1}}\min \left(B^{2},\left\|\frac{c}{r_{0}r_{1}}\right\|^{-1}\right)
\left(Br_0^{1/2}r_1^{3/2}+
B r_0^{5/4}r_1^{1/2}\log r_0\right)\\
&+\frac{1}{r_{0}r_{1}}\sum_{\substack{c=1\\ 
\gcd(c,r_0)>1}
}^{r_{0}r_{1}}\min \left(B^{2},\left\|\frac{c}{r_{0}r_{1}}\right\|^{-1}\right)
B^2 (r_0r_1)^{1/2}
\end{align*}
The first term is 
\begin{align*}
&\ll \sum_{c=1}^{r_0r_1-1} \frac{1}{c_0}
\left(Br_0^{1/2}r_1^{3/2}+
B r_0^{5/4}r_1^{1/2}\log r_0\right)\\
&\ll \left(r_0^{1/2}r_1^{3/2}+
r_0^{5/4}r_1^{1/2}\right)B(\log B)^2,
\end{align*}
since $PQ\leq B^2$ by \eqref{eq:PQ}.
The second term is 
\begin{align*}
&\ll 
\frac{B^4}{(r_{0}r_{1})^{1/2}}
+
B^2 (r_0r_1)^{1/2}
\sum_{\substack{c=1\\ \gcd(c,r_0)>1}}^{r_{0}r_{1}-1}\frac{1}{c}.
\end{align*}
But $r_{0}=p p'$ and so 
\[
\sum_{\substack{c=1\\ \gcd(c,r_{0})>1}}^{r_{0}r_{1}-1}\frac{1}{c}\leq \frac{1}{p'}\sum_{c'=1}^{r_1p-1}\frac{1}{c'}+\frac{1}{p}\sum_{c''=1}^{r_1p'-1}\frac{1}{c''}\ll\frac{(\log r_{0}r_1)^2}{r_{0}^{1/2}}.
\]
We conclude that
$$
C(r_0r_1)
\ll 
 \left(r_0^{1/2}r_1^{3/2}+
r_0^{5/4}r_1^{1/2}+\frac{B^3}{(r_0r_1)^{1/2}}+ Br_1^{1/2} \right)B(\log B)^2
$$
We now recall that $r_0\asymp P^2$ and $r_1\asymp Q^2$. 
This readily yields
\begin{align*}
C(r_0r_1)
\ll 
 \left(BPQ^3+ BP^{5/2}Q+\frac{B^4}{PQ}+ B^2Q\right)(\log B)^2.
\end{align*}
When  $P,Q$ are constrained to satisfy \eqref{eq:PQ} it is clear that 
$B^2Q\leq BPQ^3$. The statement of the  lemma is now obvious.
\end{proof}

\begin{lem}\label{lem:2}
Assume that $P,Q$ satisfy \eqref{eq:PQ}. Then 
\begin{align*}
\frac{\log B}{Q\#\mathcal{P}^{2}}
\sum_{\substack{p,p'\in\mathcal{P}\\p\neq p'}}\left|C(pp')\right|\ll 
\frac{B^4 (\log B)^2}{PQ}
\end{align*}
and 
\begin{align*}
\frac{\log B}{P\#\mathcal{Q}^{2}}
\sum_{\substack{q,q'\in\mathcal{Q}\\q\neq q'}}\left|C(qq')\right|\ll 
\frac{B^4 (\log B)^2}{PQ}.
\end{align*}
\end{lem}

\begin{proof}
We apply  Lemma \ref{lem:bianco1} in 
\eqref{eq:rosso} to obtain
\begin{align*}
C(pp')
&\ll\frac{B^2 (pp')^{1/2}}{pp'}\left( B^2+ \sum_{c=1}^{pp'-1}
\frac{pp'}{c}\right)\ll \frac{B^4}{P}+ B^2 P \log P,
\end{align*}
since $p,p'\asymp P$.
Since $P\leq B$ in \eqref{eq:PQ} we see that the 
$B^2 P \log P\ll (B^4\log B)/P$ and 
the first part of the lemma easily follows. The  second part is  similar.
\end{proof}

\begin{lem}\label{lem:3}
Assume that $P,Q$ satisfy \eqref{eq:PQ}. Then 
$$
\frac{1}{\#\mathcal{A}^{2}}|E(\mathcal{P})|\ll 
\frac{\log B}{Q}
\left(
\frac{B^{4}}{PQ} + B^2P^2\right)
$$
and 
$$
\frac{1}{\#\mathcal{A}^{2}}|E(\mathcal{Q})|\ll 
\frac{\log B}{P}
\left(
\frac{B^{4}}{PQ} + B^2Q^2\right).
$$
\end{lem}

\begin{proof}
We prove the first estimate, the second following by symmetry. 
Recall from \eqref{eq:EE} that 
\begin{align*}
E(\mathcal{P})
&=\sum_{q\in\mathcal{Q}}\sum_{\substack{p,p'\in \mathcal{P}\\ p\neq p'}}\sum_{\substack{m\\  q\mid m}}\omega (m)\left(\frac{m}{pp'}\right)\\
&=\sum_{q\in\mathcal{Q}}\sum_{\substack{p,p'\in \mathcal{P}\\ p\neq p'}}\sum_{\substack{|x|\leq B^{2},|u_{1}|,|u_{2}|\leq B\\q|x^{n}+xf(u_{1},u_{2})+g(u_{1},u_{2})}}\left(\frac{x^{n}+xf(u_{1},u_{2})+g(u_{1},u_{2})}{pp'}\right).
\end{align*}
Breaking the $x$-sum into residue classes modulo $q$, we obtain 
\[
E(\mathcal{P})=\sum_{q\in\mathcal{Q}}\sum_{\substack{p,p'\in \mathcal{P}\\ p\neq p'}}\sum_{|u_{1}|,|u_{2}|\leq B}
\hspace{-0.3cm}
\sum_{\substack{\alpha\in\mathbb{F}_{q}\\ \alpha^{n}+\alpha f(u_{1},u_{2})+g(u_{1},u_{2})=0\bmod q}}
\hspace{-0.3cm}
D(\alpha,\u),
\]
where
\[
D(\alpha,\u)=\sum_{\substack{m\in \ZZ\\ 
|\alpha+mq|\leq B^2}} 
\left(\frac{(\alpha+mq)^{n}+(\alpha+mq)f(\u)+g(\u)}{pp'}\right).
\]
Breaking the $m$-sum into residue classes modulo $pp'$, we obtain 
\begin{align*}
D(\alpha,\u)=\sum_{\beta=1}^{pp'}\left(\frac{(\alpha+\beta q)^{n}+(\alpha+\beta q)f(\u)+g(\u)}{pp'}\right)L(\beta),
\end{align*}
where $L(\beta)$ is the number of $m\in \ZZ$ for which  
$|\alpha+mq|\leq B^2$ and $m\equiv \beta \bmod{pp'}$. 
Clearly 
$$ 
L(\beta)=\frac{B^{2}}{pp' q}+O(1).
$$
Observing that $q$ is coprime to $pp'$, it therefore  follows from  
 Lemma \ref{lem : riem} that 
\begin{align*}  
D(\alpha,\u)
&\ll \frac{B^{2}}{(pp')^{1/2} q} + pp'.
\end{align*}
Since $pp'\asymp P^2$ and $q\asymp Q$ it now  easily follows 
$$
E(\mathcal{P})
\ll 
\sum_{q\in\mathcal{Q}}\sum_{\substack{p,p'\in \mathcal{P}\\ p\neq p'}}
\left(
\frac{B^{4}}{PQ} + B^2P^2\right).
$$
The statement of the lemma is now obvious. 
\end{proof}

It is finally time to combine Lemmas \ref{lem:1}--\ref{lem:3}
in \eqref{eq : sieve} and optimise our choice of parameters $P,Q$, in order to complete the proof of Theorem \ref{thm : dp1}. Recalling that $Q\leq P$ in \eqref{eq:PQ}, we
deduce that 
\begin{align*}
N(S;B)&\ll  \left(\frac{B^4}{PQ}
+
 \frac{B^2P^2}{Q}
+BPQ^3+ BP^{5/2}Q\right)(\log B)^2.
\end{align*}
The statement of Theorem \ref{thm : dp1} follows on taking  $P=B^{3/5}$ and $Q=B^{9/20}$, and noting that these values clearly satisfy the constraints outlined in \eqref{eq:PQ}.

\section{Estimation of the key character sum}\label{s:Wp}

Let $p$ be a prime such that $p\nmid 2n\Delta_{f,g}$. 
For 
any  $\lambda\in \FF_p$ and $\h,\boldsymbol{\mu}\in \FF_p^2$ we recall that the 
exponential sum in \eqref{eq:Wp} is defined to be 
\begin{align*}
W_p(\lambda,\mathbf{h},\boldsymbol{\mu})=
\hspace{-0.3cm}
\sum_{\alpha,\beta,s_{1},s_{2}\bmod p}&\left(\frac{\alpha^{n}+\alpha f(\mathbf{s}+\mathbf{h})+ g(\mathbf{s}+\mathbf{h})}{p}\right)\\
& \times \left(\frac{\beta^{n}+\beta f(\mathbf{s})+ g(\mathbf{s})}{p}\right)e_{p}(\lambda(\alpha-\beta)+\boldsymbol{\mu}\cdot\mathbf{s}).
\end{align*}
Here $f,g\in\ZZ[S_{1},S_{2}]$ are two homogeneous polynomials of degrees $2n-2$ and $2n$ respectively, such that  $n$ is odd and $g$ is separable, and such that 
\[
Y^{2}=X^{n}+Xf(S_{1},S_{2})+g(S_{1},S_{2})
\]
defines a  smooth surface in $\mathbb{P}(n,2,1,1)$.  Our assumption that $p\nmid 2n\Delta_{f,g}  
$ ensures that the  reduction modulo $p$ is  also smooth and that the reduction modulo $p$ of $g$  
is separable.  Our task in this section is to establish Proposition 
\ref{prop : expsum}.  The estimate $$
W_p(\lambda,\mathbf{h},\boldsymbol{\mu})=O(p^3)$$  
is an easy consequence of Lemma 
\ref{lem : riem}, which therefore handles the case 
 $\h=\boldsymbol{\mu}=\0$. It remains to prove the following result. 
 
 \begin{prop}
 Let $p\equiv 2\bmod{n}$ and let $p\nmid 2n\Delta_{f,g}$. 
Assume that 
$\lambda\in \FF_p^\times $ and $\h,\boldsymbol{\mu}\in \FF_p^2$, with 
$(\mathbf{h},\boldsymbol{\mu})\neq(\boldsymbol{0},\boldsymbol{0})$.
Then there exists a constant $C_n>0$ such that 
\[
|W_p(\lambda,\mathbf{h},\boldsymbol{\mu})|\leq C_n p^{5/2}.
\]
\label{prop:milan}
\end{prop}

The same estimate holds for 
$W_p(\lambda,\mathbf{h},\boldsymbol{\mu})
$ for any prime $p$, but the restriction 
$p\equiv 2\bmod{n}$ makes the proof notationally less cumbersome.
We have not been able to apply existing results in the literature to deduce 
Proposition \ref{prop:milan}.
However, after some preliminary manoeuvres we shall bring the sum into a form that can be handled by work of Katz \cite[Theorem $4$]{Kat99}. Unfortunately,  
as we shall discuss in
Section \ref{s:katz},  the relevant varieties are too  singular 
to extract  any improvement over the  bound $O(p^3)$ from Katz. 
  Thus we shall adopt an alternative course of action to arrive at Proposition \ref{prop:milan}. 

Since  $\lambda\neq 0$, our first move is to observe that 
\[
W_p(\lambda,\mathbf{h},\boldsymbol{\mu})=\sum_{\substack{\x=(u,v,x,y,s_{1},s_{2})\in\mathbb{F}_{p}^{6}\\G_{1}(\mathbf{x})=G_{2}(\mathbf{x})=0}}e_{p}(\lambda(x-y)+\boldsymbol{\mu}\cdot\mathbf{s}),
\] 
for polynomials $G_1,G_2\in \FF_p[U,V,X,Y,S_1,S_2]$ given by
\begin{align*}
G_{1}
&=-U^{2}+X^{n}+ X f(S_{1},S_{2})+g(S_{1},S_{2}),\\
 G_{2}&=-V^{2}+Y^{n}+Y f(S_{1}+h_{1},S_{2}+h_{2}) +g(S_{1}+h_{1},S_{2}+h_{2}) 
.
 \end{align*}
It will be more convenient to transform 
$W_p(\lambda,\mathbf{h},\boldsymbol{\mu})
$ into  a sum in which the monomials involving 
$U,V,X,Y$  have degree $2n$.
This is achieved in the  following result.

\begin{lem}\label{lem:lecce}
Assume that $p\equiv 2\bmod{n}$ and 
let $\gamma\in\mathbb{F}_{p}^{\times}$ be a non-square.
Then 
we have 
$$
W_p(\lambda,\mathbf{h},\boldsymbol{\mu})=
\frac{1}{4}\sum_{i,j\in \{0,1\}}
W_{p,i,j}(\lambda,\mathbf{h},\boldsymbol{\mu}),
$$
where
$$
W_{p,i,j}(\lambda,\mathbf{h},\boldsymbol{\mu})=\sum_{\substack{\x
\in\mathbb{F}_p^{6}\\G_1^{(i)}(\mathbf{x})=G_2^{(j)}(\mathbf{x})=0}}e_{p}(\lambda(
\gamma^i x^2-\gamma^jy^2)+ \boldsymbol{\mu}\cdot\mathbf{s}),
$$

for $i,j\in \{0,1\}$, with
\begin{align*}
&G_1^{(i)}=-U^{2n}+\gamma^{ni} X^{2n}+ \gamma^{i} X^{2} f(S_{1},S_{2})+g(S_{1},S_{2}),
\\& 
G_2^{(j)}=-V^{2n}+\gamma^{nj}Y^{2n}+\gamma^j Y^2
f(S_{1}+h_{1},S_{2}+h_{2})+g(S_{1}+h_{1},S_{2}+h_{2}).
\end{align*}
\end{lem}

\begin{proof}
Since $p\equiv 2 \bmod{n}$ we have that $\gcd(n,p-1)=1$ and then every element of $\FF_p$ is a $n$-th power in $\FF_p$. Next, recall that $\mathbb{F}_{p}^{\times}/{\mathbb{F}_{p}^{\times}}^2=\{\pm 1\}$ and let $\gamma\in\mathbb{F}_{p}^{\times}$ be a non-square. If $\alpha\in\mathbb{F}_{p}$, then there exists $a\in\mathbb{F}_{p}$ such that 
either $\alpha=a^{2}$ or $\alpha=\gamma a^{2}$.  The statement of the lemma is now clear.
\end{proof}

\subsection{Comparison with work of Katz}\label{s:katz}

In this short section we take a moment to check what comes out of applying general work by Katz \cite{Kat99}
 on singular exponential sums. 
 We can recognise our exponential sum $W_{p,i,j}(\lambda,\mathbf{h},\boldsymbol{\mu})$ as the exponential sum considered in 
 \cite[Theorem $4$]{Kat99}.
 Let $ X\subset \PP_{\FF_q}^6$ be the geometrically integral complete intersection  
$$
\begin{cases}
0=-U^{2n}+\gamma^{ni} X^{2n}+ \gamma^{i} X^{2} f(S_{1},S_{2})+g(S_{1},S_{2}),\\ 
0=-V^{2n}+\gamma^{nj}Y^{2n}+\gamma^j Y^2
f(S_{1}+h_{1}T,S_{2}+h_{2}T)+g(S_{1}+h_{1}T,S_{2}+h_{2}T). 
\end{cases}
$$
Let   $L$ be the hyperplane $T=0$ and let  $H$ be the hypersurface 
$$
\lambda(\gamma^i X^{2}-\gamma^jY^{2})+(\boldsymbol{\mu}\cdot\mathbf{S})T=0.
$$
Then 
$W_{p,i,j}(\lambda,\mathbf{h},\boldsymbol{\mu})$ precisely  matches the exponential sum considered in  \cite[Theorem $4$]{Kat99}, with $V=X[1/T]$  and $f$ being given by the function 
$$
\left(\lambda(\gamma^i X^{2}-\gamma^jY^{2})+(\boldsymbol{\mu}\cdot\mathbf{S})T\right)/T^2.
$$
Let $\delta=\dim \sing(X\cap L\cap H)$. In the most favourable situation, when
$\delta\geq \dim \sing(X\cap L)$, it follows from 
\cite[Theorem $4$(1)]{Kat99} that 
$$
W_{p,i,j}(\lambda,\mathbf{h},\boldsymbol{\mu})\ll p^{(5+\delta)/2}.
$$
But $X\cap L\cap H$ is cut out by the system
$$
\begin{cases}
0=-U^{2n}+X^{2n}+ X^{2} f(S_{1},S_{2})+g(S_{1},S_{2}),\\
0=-V^{2n}+Y^{2n}+ Y^{2} f(S_{1},S_{2})+g(S_{1},S_{2}),\\
0=\gamma^i X^{2}-\gamma^jY^{2},\\
0=T.
\end{cases}
$$
The  Jacobian matrix of this system is
\[
\begin{pmatrix}
-2nU^{2n-1} & 0 & 0 & 0\\
0 & -2nV^{2n-1} & 0 & 0\\
2nX^{2n-1} +2Xf & 0 & 2\gamma^i X & 0 \\
0 & 2nY^{2n-1} +2Yf & -2 \gamma^j Y & 0\\
\frac{\partial g}{\partial S_{1}}+X^{2}\frac{\partial f}{\partial S_{1}}& \frac{\partial g}{\partial S_{1}} +Y^{2}\frac{\partial f}{\partial S_{1}}& 0 & 0\\
\frac{\partial g}{\partial S_{2}}+X^{2}\frac{\partial f}{\partial S_{2}}& \frac{\partial g}{\partial S_{2}} + Y^{2}\frac{\partial f}{\partial S_{2}}& 0 & 0\\
0 & 0 & 0 & 1
\end{pmatrix}.
\]
The set of points such that the third column vanishes satisfies the system
\[
\begin{cases}
U^{2n}=g(S_{1},S_{2}),\\
V^{2n}=g(S_{1},S_{2}),\\
X=Y=T=0,
\end{cases}
\]
which has dimension $1$. Thus $\delta\geq 1$ and 
\cite[Theorem $4$]{Kat99} will not yield an improvement over the bound $O(p^3).$

\subsection{Strategy for proving Proposition \ref{prop:milan} }

The goal of this section is to prove Proposition \ref{prop:milan}, subject to  
an estimate for the dimension of the singular locus of a certain variety that will be 
examined in the next section. 
We shall  always assume 
that $\lambda\neq 0$ and  
$(\mathbf{h},\boldsymbol{\mu})\neq(\boldsymbol{0},\boldsymbol{0})$.

We assume that $p\equiv 2\bmod{n}$ and 
$p\nmid 2n\Delta_{f,g}$. 
Let $\gamma\in\mathbb{F}_{p}^{\times}$ be a non-square.
Our first move is an application of Lemma \ref{lem:lecce}, rendering it sufficient to study
$W_{p,i,j}(\lambda,\mathbf{h},\boldsymbol{\mu})$, for $i,j\in \{0,1\}$.
We shall  apply a method of Hooley \cite{hooleysum} to do so, the outcome of which we have  
recorded as  Theorem~\ref{thm:hooley} in the appendix. This requires us to estimate 
\[
\begin{split}
N(\tau)&=\#\left\{\mathbf{x}=(u,v,x,y,s_{1},s_{2}) 
\in\mathbb{F}_{q}^{6}: 
\begin{array}{l}
G_1^{(i)}(\mathbf{x})=G_2^{(j)}(\mathbf{x})=0\\
\lambda (\gamma^ix^2-\gamma^jy^2)+\boldsymbol{\mu}\cdot\mathbf{s}=\tau
\end{array}
\right\},
\end{split}
\]
where  $q=p^r$. 
Since $Y$ only appears to even degree in $G_2^{(j)}$, we can eliminate $y$ from  
$G_2^{(j)}(\mathbf{x})$ by writing  
$\gamma^jy^{2}=\gamma^i x^{2}+ \overline\lambda   
\boldsymbol{\mu}\cdot\mathbf{s}-\overline{\lambda} \tau$, 
 where $\overline\lambda$ is the inverse of $\lambda$ in $\FF_q^\times$. 
This yields 
$$
N(\tau)
=\#\left\{\mathbf{y}=(u,v,x,s_{1},s_{2})\in\mathbb{F}_{q}^{5}: G_{1}^{(i)}(\mathbf{y})=G_\tau^{(i)}(\mathbf{y})=0\right\},
$$
where 
\begin{align*}
G_\tau^{(i)}
&=-V^{2n}+
 (\gamma^i X^2+\overline{\lambda}(\boldsymbol{\mu}\cdot\mathbf{S})-\overline{\lambda}\tau)^{n}\\
 &\quad
 + (\gamma^i X^2+\overline{\lambda}(\boldsymbol{\mu}\cdot\mathbf{S})-\overline{\lambda}\tau) f(S_{1}+h_{1},S_{2}+h_{2})
  +g(S_{1}+h_{1},S_{2}+h_{2}). 
\end{align*}
The polynomial $G_1^{(i)}$ is homogenous of degree $2n$. 
We shall need to  homogenise the polynomial $G_\tau^{(i)}$, which 
we do by introducing a sum over $t\in \FF_q^\times$ and making an obvious change of variables.  The resulting polynomial $H_{\tau}^{(i)}\in \FF_q[V,X,S_1,S_2,T]$ is given by 
\begin{align*}
H_{\tau}^{(i)}
&=-V^{2n}+
 (\gamma^i X^2+\overline{\lambda}T(\boldsymbol{\mu}\cdot\mathbf{S})-\overline{\lambda}\tau T^2)^{n}\\
 &\quad
 + (\gamma^i X^2+\overline{\lambda}T(\boldsymbol{\mu}\cdot\mathbf{S})-\overline{\lambda}\tau T^2) f(S_{1}+h_{1}T,S_{2}+h_{2}T)\\
&\quad  +g(S_{1}+h_{1}T,S_{2}+h_{2}T).  
\end{align*}
To ease notation we henceforth suppress 
the index $i$ from our notation, setting $G_1^{(i)}=G$ and $H_\tau^{(i)}=H_\tau$.
We are now led  to the expression 
$$
N(\tau)=\frac{1}{q-1}\left(N_1(\tau)- N_2(\tau)\right), 
$$
where
$$
N_1(\tau)
=\#\left\{(\mathbf{y},t)\in \mathbb{F}_{q}^{6}: G(\mathbf{y})=H_{\tau}(\mathbf{y},t)=0\right\},
$$
and 
$$
N_2(\tau)
=\#\left\{\mathbf{y}\in\mathbb{F}_{q}^{5}: G(\mathbf{y})=H_{\tau}(\mathbf{y},0)=0\right\},
$$
where
$\y=(u,v,x,s_{1},s_{2})$, as before. 

In order to estimate $N_1(\tau)$ we shall need to know about the singular locus of the complete intersection cut out by the two polynomials $G$ and $H_\tau$.
This is summarised in the following result. 

\begin{prop} 
Assume that
 $\lambda\neq 0$ and  
$(\mathbf{h},\boldsymbol{\mu})\neq(\boldsymbol{0},\boldsymbol{0})$. 
For all but at most $32n^5$ choices of $\tau\in \FF_q$, the equations 
$G=H_\tau=0$ cut out a complete intersection of codimension $2$ in
$\PP_{\FF_q}^5$,
with isolated singularities. 
\label{prop:tau}
\end{prop}

Taking this result on faith for the moment, let us see how it suffices to complete the proof of Proposition \ref{prop:milan}. 
Appealing to Theorem 1 in the appendix by Katz to  \cite{HK91}, it follows from Proposition \ref{prop:tau} that there exists a set $U\subset \FF_q$, with $\#U\leq 32n^5$, such that 
\[
\frac{N_1(\tau)}{q-1}= q^{3} + O(q^{2})
\]
for all $\tau\not \in U$. When $\tau\in U$ we invoke the Lang--Weil estimate to deduce that 
$$
\frac{N_1(\tau)}{q-1}= q^{3} + O(q^{5/2}).
$$
 The implied constants in both of these estimates depend only on $n$.
On the other hand, the 
variety
$G=H_\tau=T=0$ 
has codimension $3$ in
$\PP_{\FF_q}^5$. Thus 
$$
\frac{N_2(\tau)}{q-1}=O(q^2),
$$ 
by the Lang--Weil estimate, for a further implied constant that depends only on $n$.
Hence
\[
\begin{split}
\sum_{\tau\in\mathbb{F}_q} \left| 
N(\tau)-q^3\right|^{2}
&\ll\sum_{\tau\not \in U}q^4 +\sum_{ \tau\in U} q^5\ll q^5,
\end{split}
\]
for an implied constant depending only on $n$.  Theorem \ref{thm:hooley} 
now yields
$$
W_{p,i,j}(\lambda,\mathbf{h},\boldsymbol{\mu})\ll p^{5/2},
$$ 
for $i,j\in \{0,1\}$.
Once inserted into Lemma \ref{lem:lecce}, this therefore completes the proof of 
Proposition \ref{prop:milan} subject to a verification of Proposition \ref{prop:tau}.

\section{The singular locus}

This section is devoted to proving   Proposition \ref{prop:tau}.
Since we are working over $\overline{\mathbb{F}}_{q}$, 
without loss of generality we may assume that $i=0$. Thus 
$$
G=-U^{2n}+X^{2n}+ X^{2} f(S_{1},S_{2})+g(S_{1},S_{2}) 
$$
and 
\begin{align*}
H_{\tau}
&=-V^{2n}+
 (X^2+\overline{\lambda}T(\boldsymbol{\mu}\cdot\mathbf{S})-\overline{\lambda}\tau T^2)^{n}\\
 &\quad
 + ( X^2+\overline{\lambda}T(\boldsymbol{\mu}\cdot\mathbf{S})-\overline{\lambda}\tau T^2) f_{\h}(S_{1},S_{2},T)
+g_{\h}(S_{1},S_{2},T),
\end{align*}
where 
$$
f_{\mathbf{h}}(S_{1},S_{2},T)=f(S_{1}+h_{1}T,S_{2}+h_{2}T),
$$
and similarly for $g_{\h}$. Let us denote by $V_\tau\subset \PP_{\FF_q}^5$ the variety cut out by the equations $G=H_\tau=0$. It is clearly a complete intersection of codimension $2$.
Our task is to show that $\dim \sing(V_\tau)=0$, for all but at most $32n^5$ choices of  $\tau\in \FF_q$.

The Jacobian $J_{\tau}$ of $V_{\tau}$ is given by the matrix
$$
\begin{pmatrix}
-2nU^{2n-1} & 0 &
2nX^{2n-1} +2Xf
&
X^{2}\frac{\partial f}{\partial S_{1}}+\frac{\partial g}{\partial S_{1}} &
X^{2}\frac{\partial f}{\partial S_{2}}+\frac{\partial g}{\partial S_{2}} &
0\\
0 &
-2nV^{2n-1} 
 & \frac{\partial H_{\tau}}{\partial X}
 & \frac{\partial H_{\tau}}{\partial S_{1}}
 & \frac{\partial H_{\tau}}{\partial S_{2}}
 & \frac{\partial H_{\tau}}{\partial T}
\end{pmatrix},
$$
where 
\[
\begin{split}
\frac{\partial H_{\tau}}{\partial T}&= n\overline{\lambda}(\boldsymbol{\mu}\cdot\mathbf{S}-2\tau T)( X^{2}+\overline{\lambda}T(\boldsymbol{\mu}\cdot\mathbf{S})-\overline{\lambda}\tau T^{2})^{n-1}+\overline{\lambda}(\boldsymbol{\mu}\cdot\mathbf{S}-2\tau T)f_{\mathbf{h}}\\&\quad+( X^{2}+\overline{\lambda}T(\boldsymbol{\mu}\cdot\mathbf{S})-\overline{\lambda}\tau T^{2}) \frac{\partial f_{\mathbf{h}}}{\partial T}+ \frac{\partial g_{\mathbf{h}}}{\partial T}.
\end{split}
\]
A point  $[\mathbf{y},t]\in\PP_{\FF_q}^{5}$ belongs to  $\sing (V_{\tau})$ if and only if  $[\mathbf{y},t]\in V_{\tau}$ and  the matrix $J_\tau$ has rank at most $1$  when evaluated at the vector $(\y,t)$,
where we recall that $\y=(u,v,x,s_1,s_2)$. 
Note that any point $[\y,t]\in \PP_{\FF_q}^5$ with $u=x=s_1=s_2=0$ lies in 
$\sing(V_\tau)$ if and only if
$$
v^{2n}=(-\bar\lambda^n\tau^n- \bar \lambda \tau f(\h)+g(\h))t^{2n}.
$$
Thus $\dim \sing(V_\tau)\geq 0$. It is sufficient to show that 
$\dim \sing(V_\tau)\leq 0$ for all but at most $32n^5$ values of $\tau$.

We  proceed by breaking  $\sing(V_\tau)\subset \PP_{\FF_q}^5$ into three subsets
$K_\tau^{(1)}\cup K_\tau^{(2)}\cup L_\tau$. Here, $K_\tau^{(1)}$ is the set of $[\y,t]\in \sing(V_\tau)$ for which the first row in $J_\tau$ vanishes. In other words $K_\tau^{(1)}$ is cut out by the system of equations 
\begin{equation}\label{eq:system-K1}
\begin{cases}
G=H_{\tau}=0\\
U=
nX^{2n-1} +Xf=
X^{2}\frac{\partial f}{\partial S_{1}}+\frac{\partial g}{\partial S_{1}} =
X^{2}\frac{\partial f}{\partial S_{2}}+\frac{\partial g}{\partial S_{2}} =0.
\end{cases}
\end{equation}
Likewise, 
 $K_\tau^{(2)}$ is the set of $[\y,t]\in \sing(V_\tau)$ for which the second row in $J_\tau$ vanishes, so that it is cut out by the system of equations 
\begin{equation}\label{eq:system-K2}
\begin{cases}
G=H_{\tau}=0\\
V=\frac{\partial H_\tau}{\partial X}=\frac{\partial H_\tau}{\partial S_{1}}=
\frac{\partial H_\tau}{\partial S_{2}}=
\frac{\partial H_\tau}{\partial T}=0.
\end{cases}
\end{equation}
We shall prove the following result.

\begin{lem}\label{lem:K}
We have $\dim (K_\tau^{(i)})\leq 0$ for $i=1,2$.
\end{lem}

Finally, let 
 $L_\tau$ be the set of $[\y,t]\in \sing(V_\tau)$ for which neither row vanishes in $J_\tau$ vanishes.
 Thus it is cut out by the system of equations 
\begin{equation} 
\begin{cases}
0=G=H_\tau,\\  
0=\frac{\partial H_\tau}{\partial T},\\
0=\frac{\partial G}{\partial S_{1}}\cdot \frac{\partial H_\tau}{\partial S_{2}}-\frac{\partial G}{\partial S_{2}}\cdot \frac{\partial H_\tau}{\partial S_{1}},
\\
0=U=V.
\end{cases}
\label{eq : sys0}
\end{equation}
We shall prove the following result. 

\begin{lem}\label{lem:L}
We have $\dim (L_\tau)\leq 0$, 
for all but at most $32n^5$ choices of $\tau\in \FF_q$.
\end{lem}

Taken together, Lemmas \ref{lem:K} and \ref{lem:L} complete the proof of 
Proposition \ref{prop:tau}.

\begin{proof}[Proof of Lemma \ref{lem:K}] 
We start
 by considering $K_\tau^{(2)}$, which we wish to show has dimension at most $0$. 
  To this end it is enough to show that it does not intersect the hyperplane $T=0$.
In view of \eqref{eq:system-K2}, 
 the  point  $[\y,0]$ lies on $K_\tau^{(2)}$ if and only if 
\[ 
\begin{cases}
0=v=t,\\
0=-u^{2n}+x^{2n}+  x^{2} f(s_{1},s_{2})+g(s_{1},s_{2})\\ 0=x^{2n}+x^{2} f(s_{1},s_{2}) +g(s_{1},s_{2}),\\
0=2nx^{2n-1} +2xf(s_{1},s_{2}),\\
0=x^{2}\frac{\partial f}{\partial S_{1}}(s_{1},s_{2})+\frac{\partial g}{\partial S_{1}}(s_{1},s_{2}),\\
0=x^{2}\frac{\partial f}{\partial S_{2}}(s_{1},s_{1})+\frac{\partial g}{\partial S_{2}}(s_{1},s_{2}),\\
0=n\overline{\lambda}(\boldsymbol{\mu}\cdot\mathbf{s}) x^{2n-2}+\overline{\lambda}(\boldsymbol{\mu}\cdot\mathbf{s})f_{\mathbf{h}}(s_{1},s_{2},0)+ x^{2} \frac{\partial f_{\mathbf{h}}}{\partial T}(s_{1},s_{2},0)+ \frac{\partial g_{\mathbf{h}}}{\partial T}(s_{1},s_{2},0).
\end{cases}
\]
Since $p\nmid 2n\Delta_{f,g}$, it also follows that the plane curve  
$$
X^{2n}+X^{2}f(S_{1},S_{2})+g(S_{1},S_{2})=0
$$ in $\PP_{\FF_q}^{2}$ is  smooth. Hence the 3rd, 4th, 5th and 6th equation together imply that 
$x=s_{1}=s_{2}=0$. But then the 2nd equation implies that $u=0$ and this proves that the intersection of $K_\tau^{(2)}$ with the hyperplane $T=0$ is empty, whence $\dim (K_\tau^{(2)})\leq 0$.
We obtain $\dim (K_\tau^{(1)})\leq 0$ by repeating the same argument 
with \eqref{eq:system-K1}  
and 
switching the role of $u$ and $v$.
\end{proof}

\begin{proof}[Proof of Lemma \ref{lem:L}] 
We begin by showing that  $L_\tau$ has finitely many points  with $t=0$.
When $T=0$, the system \eqref{eq : sys0} becomes
\[ 
\begin{cases}
0=X^{2n}+  X^{2} f(S_{1},S_{2})+g(S_{1},S_{2}),\\
0=n\overline{\lambda}(\boldsymbol{\mu}\cdot\mathbf{S})X^{2n-2}+\overline{\lambda}(\boldsymbol{\mu}\cdot\mathbf{S})f+X^{2} \frac{\partial f_{\mathbf{h}}}{\partial T}(S_{1},S_{2},0)+ \frac{\partial g_{\mathbf{h}}}{\partial T}(S_{1},S_{2},0),\\
0=T=U=V.
\end{cases}
\]
Note that for any binary form $F\in \FF_q[S_1,S_2]$  we have 
\[
\begin{split}
\frac{\partial G}{\partial T}(S_{1}+h_{1}T,S_{2}+h_{2}T),
&=\sum_{i=1,2}h_{i}\frac{\partial G}{\partial S_{i}}(S_{1}+h_{1}T,S_{2}+h_{2}T),\\
&=(\mathbf{h}\cdot\nabla G)(S_{1}+h_{1}T,S_{2}+h_{2}T).
\end{split}
\]
Thus the  system becomes
\[ 
\begin{cases}
0=X^{2n}+  X^{2} f(S_{1},S_{2})+g(S_{1},S_{2}),\\
0=n\overline{\lambda}(\boldsymbol{\mu}\cdot\mathbf{S})X^{2n-2}+\overline{\lambda}(\boldsymbol{\mu}\cdot\mathbf{S})f+ X^{2}(\mathbf{h}\cdot\nabla f)(S_{1},S_{2})+ (\mathbf{h}\cdot\nabla g)(S_{1},S_{2}),\\ 
0=T=U=V.
\end{cases}
\]
If   $\boldsymbol{\mu}\neq\boldsymbol{0}$ the monomial $(\boldsymbol{\mu}\cdot\mathbf{S})X^{2n-2}$ does not vanish identically. If $\boldsymbol{\mu}=\boldsymbol{0}$, then $\mathbf{h}\cdot\nabla g$ does not vanish identically, since 
then $\mathbf{h}\neq \boldsymbol{0}$. 
  Thus the 2nd equation involves 
a non zero polynomial of degree $2n-1$ in $X,S_{1},S_{2}$. 
On the other hand, the 1st equation defines an irreducible form of  degree $2n$, 
and so the system meets in at most $4n^{2}-2n$ points by B\'ezout's theorem.

Let us now count the solutions of system \eqref{eq : sys0} with $t\neq 0$. We shall introduce a further variable 
$
Z=\overline{\lambda}(\mu\cdot \mathbf{S}-\tau T),
$
leading us to study the  system
\begin{equation} 
\begin{cases}
0=X^{2n}+  X^{2} f(S_{1},S_{2})+g(S_{1},S_{2}),\\
0=(X^{2}+TZ)^{n}+  (X^{2}+TZ) f_{\h}
+g_{\h},\\  
0=n(2Z-\overline{\lambda}(\boldsymbol{\mu}\cdot\mathbf{S}))(X^{2}+TZ)^{n-1}+(2Z-\overline{\lambda}(\boldsymbol{\mu}\cdot\mathbf{S}))f_{\h}
\\
\qquad +(X^{2}+TZ)\frac{\partial f_{\mathbf{h}}}{\partial T}+\frac{\partial g_{\mathbf{h}}}{\partial T},\\
0=(X^{2}\frac{\partial f}{\partial S_{1}}+\frac{\partial g}{\partial S_{1}})((X^{2}+TZ)\frac{\partial f_{\h}}{\partial S_{2}}+\frac{\partial g_{\h}}{\partial S_{2}})\\
\qquad -(X^{2}\frac{\partial f}{\partial S_{2}}+\frac{\partial g}{\partial S_{2}})((X^{2}+TZ)\frac{\partial f_{\mathbf{h}}}{\partial S_{1}}+\frac{\partial g_{\mathbf{h}}}{\partial S_{1}}),\\
0=Z-\overline{\lambda}(\boldsymbol{\mu}\cdot \mathbf{S}-\tau T) .
\end{cases}
\label{eq : sys2}
\end{equation}
Let $V\subset\PP_{\FF_q}^{4}$ be the zero set of the first four equations. This variety 
does not depend on $\tau$.  We shall prove below in Proposition \ref{prop : dim1} that 
$\dim(V)=1$.

Let 
$V_{1},\dots ,V_{k}$ be the irreducible components of $V$
 that are not contained in the hyperplane $T=0$. 
 Using the form of B\'ezout's theorem found in Example 8.4.6 of Fulton \cite{fulton}, one can bound $k$ by the product of the degrees of the forms defining $V$ in $\PP_{\FF_q}^4$. Thus 
 $$k\leq 4n^2(2n-1)^3\leq 32n^5.
 $$
 Moreover, Proposition \ref{prop : dim1} implies that $\dim (V_i)\leq 1$ for $1\leq i\leq k$. 
 For any $i\in \{1,\dots,k\}$ let 
$v_i=[x_{i},z_{i},s_{1,i},s_{2,i},t_{i} ]\in V_{i}$ be such that $t_{i}\neq 0$.  
Then $v_i$ lies on the hyperplane  $
Z=\overline{\lambda}(\mu\cdot \mathbf{S}-\tau T))$ if and only if
\[
z_{i}=\overline{\lambda}(\boldsymbol{\mu}\cdot\mathbf{s}_{i}-\tau t_{i}).
\]
This is  true for at most a single value of $\tau$, since $t_{i}\neq 0$. Hence for all but at most $k\leq 32n^5$  exceptional   $\tau$, we may conclude that 
$V_i$ is not contained in the hyperplane 
$Z=\overline{\lambda}(\mu\cdot \mathbf{S}-\tau T)$,  for  $i\in \{1,\dots,k\}$, 
which implies that the intersection of $V_i$ with the hyperplane 
$Z=\overline{\lambda}(\mu\cdot \mathbf{S}-\tau T)$ has dimension $0$. 
 It follows 
that the system \eqref{eq : sys2} has finitely many solutions with $t\neq 0$,  for all but at
most $32n^5$ values of  $\tau$.
\end{proof}

It remains to prove that that the dimension of $V$ is  $1$. 
For this we shall  require the following preliminary facts.

\begin{lem}
Let $F\in\mathbb{F}_{q}[S_{1},S_{2}]$ be a separable polynomial of degree $d$ and let 
 $P\subset \PP_{\FF_q}^1$ be its zero locus. 
Then
\begin{itemize}
\item[(i)] For any $[a,b],[c,d]\in P$ we have 
\[
\frac{\partial F}{\partial S_{1}}(a,b)\frac{\partial F}{\partial S_{2}}(c,d)-\frac{\partial F}{\partial S_{2}}(a,b)\frac{\partial F}{\partial S_{1}}(c,d)=0
\]
if and only if $[a,b]=[c,d]$.
\item[(ii)] For any $h_{1},h_{2},s_{1},s_{2}\in\mathbb{F}_{q}$ with $[s_{1},s_{2}]\neq [h_{1},h_{2}]$, let
$$G(T)=F(s_{1}+h_{1}T,s_{2}+h_{2}T).
$$ Then $G$ is non-constant and separable, with 
$$
\deg(G)=\begin{cases}
d &\text{ if $[h_{1},h_{2}]\not\in P$,}\\
d-1 &\text{ if $[h_{1},h_{2}]\in P$.}
\end{cases}
$$
Furthermore, if $t,t'$ are distinct roots of $G$ then 
\[
[s_{1}+h_{1}t ,s_{2}+h_{2}t]\neq [s_{1}+h_{1}t' ,s_{2}+h_{2}t'].
\]
\end{itemize}
\label{lem : polf}
\end{lem}

\begin{proof}
Without loss of generality we may assume that $P$ consists of points  
$$
[1,\alpha_{1}],
\dots ,[1,\alpha_{d}],
$$ 
for distinct  $\alpha_1,\dots,\alpha_d\in \overline\FF_q$.  It follows that $F=\prod_{i=1}^{d}(S_{1}\alpha_{i}-S_{2})$, so that 
\[
\frac{\partial F}{\partial S_{1}}=\sum_{i=1}^{d}\alpha_{i}\prod_{j\neq i}(S_{1}\alpha_{j}-S_{2}),\quad \frac{\partial F}{\partial S_{2}}=-\sum_{i=1}^{d}\prod_{j\neq i}(S_{1}\alpha_{j}-S_{2}).
\] 
If $[1,\alpha_{k}],[1,\alpha_{m}]\in P$ then
\begin{align*}
0&=\frac{\partial F}{\partial S_{1}}(1,\alpha_{k})\frac{\partial F}{\partial S_{2}}(1,\alpha_{m})-\frac{\partial F}{\partial S_{2}}(1,\alpha_{k})\frac{\partial F}{\partial S_{1}}(1,\alpha_{m})\\
&=(\alpha_{k}-\alpha_{m})\left(\prod_{j\neq k}(\alpha_{j}-\alpha_{k})\right)\left(\prod_{j\neq m}(\alpha_{j}-\alpha_{m})\right)
\end{align*}
if and only if $\alpha_{k}=\alpha_{m}$.  This establishes part (i).

Turning to  part (ii),
we first  assume that $[h_1,h_2]\not\in P$. Then $t$ is a root of $g$ 
if and only if either $s_{1}+h_{1}t=s_{2}+h_{2}t=0$ or $[s_{1}+h_{1}t ,s_{2}+h_{2}t]\in P$. Hence $t$ is a root of $g$ if and only if there exists $[1,\alpha]\in P$ and $\nu_{1},\nu_{2}\in\mathbb{F}_{p}$ such that
\begin{equation}\label{eq:gelato}
\begin{cases}
s_{1}+h_{1}t=\nu_{1}\\
s_{2}+h_{2}t=\nu_{2}\alpha .
\end{cases}
\end{equation}
But then 
\[
t=\frac{\alpha s_{1}-s_{2}}{h_{2}-\alpha h_{1}}.			
\]
Moreover, for $[1,\alpha_{1}],[1,\alpha_{2}]\in P$ we get that
\[
(\alpha_{1} s_{1}-s_{2})(h_{2}-\alpha_{2} h_{1})=(\alpha_{2} s_{1}-s_{2})(h_{2}-\alpha_{1} h_{1})
\]
if and only if $(\alpha_{1} -\alpha_{2})(s_{1}h_{2}-s_{2}h_{1})=0$, which is if and only if   $\alpha_{1}=\alpha_{2}$,  since we are assuming $[s_{1},s_{2}]\neq [h_{1},h_{2}]$. The  result follows since $f$ is of degree $d$ and separable. Suppose next that   $[h_1,h_2]\in P$. Then we  repeat the same argument, but  observe that the system \eqref{eq:gelato}
is not solvable in the case $[1,\alpha]=[h_{1},h_{2}]$.
The final claim in 
part (iii)  is a   direct consequence of  our argument.
\end{proof}

Recall that 
 $V\subset \PP_{\FF_q}^4$ is given by 
\begin{equation}\label{eq:napoli}
\begin{cases}
0=X^{2n}+  X^{2} f(S_{1},S_{2})+g(S_{1},S_{2}),\\
0=(X^{2}+TZ)^{n}+  (X^{2}+TZ) f_{\h}
+g_{\h},\\  
0=n(2Z-\overline{\lambda}(\boldsymbol{\mu}\cdot\mathbf{S}))(X^{2}+TZ)^{n-1}+(2Z-\overline{\lambda}(\boldsymbol{\mu}\cdot\mathbf{S}))f_{\h}
\\
\qquad +(X^{2}+TZ)\frac{\partial f_{\mathbf{h}}}{\partial T}+\frac{\partial g_{\mathbf{h}}}{\partial T},\\
0=(X^{2}\frac{\partial f}{\partial S_{1}}+\frac{\partial g}{\partial S_{1}})((X^{2}+TZ)\frac{\partial f_{\h}}{\partial S_{2}}+\frac{\partial g_{\h}}{\partial S_{2}})\\
\qquad -(X^{2}\frac{\partial f}{\partial S_{2}}+\frac{\partial g}{\partial S_{2}})((X^{2}+TZ)\frac{\partial f_{\mathbf{h}}}{\partial S_{1}}+\frac{\partial g_{\mathbf{h}}}{\partial S_{1}}).
\end{cases}
\end{equation}
We shall prove the following result.

\begin{prop}\label{prop : dim1}
Assume that $p\nmid 2n\Delta_{f,g}$. Then  $\dim(V)=1$.
\end{prop}

We can rewrite the final equation in \eqref{eq:napoli} as 
\begin{equation}\label{eq:venice}
(X^{2}+TZ)U_{1}=U_{2},
\end{equation}
where
\begin{equation}\label{eq:cornetto}
\begin{split}
&U_{1}=X^{2}\left(\frac{\partial f_{\mathbf{h}}}{\partial S_{2}}\frac{\partial f}{\partial S_{1}}-\frac{\partial f_{\mathbf{h}}}{\partial S_{1}}\frac{\partial f}{\partial S_{2}}\right)+\left(\frac{\partial f_{\mathbf{h}}}{\partial S_{2}}\frac{\partial g}{\partial S_{1}}-\frac{\partial f_{\mathbf{h}}}{\partial S_{1}}\frac{\partial g}{\partial S_{2}}\right),\\& U_{2}=X^{2}\left(\frac{\partial g_{\mathbf{h}}}{\partial S_{2}}\frac{\partial f}{\partial S_{1}}-\frac{\partial g_{\mathbf{h}}}{\partial S_{1}}\frac{\partial f}{\partial S_{2}}\right)+\left(\frac{\partial g_{\mathbf{h}}}{\partial S_{2}}\frac{\partial g}{\partial S_{1}}-\frac{\partial g_{\mathbf{h}}}{\partial S_{1}}\frac{\partial g}{\partial S_{2}}\right).
\end{split}
\end{equation}
We shall first prove  Proposition \ref{prop : dim1}  in the case where $f$ is identically zero. Then  $U_1=0$ and 
$$
U_2=\frac{\partial g_{\mathbf{h}}}{\partial S_{2}}\frac{\partial g}{\partial S_{1}}-\frac{\partial g_{\mathbf{h}}}{\partial S_{1}}\frac{\partial g}{\partial S_{2}}.
$$ 
Thus, $V$ is cut out by the system of equations
$$
\begin{cases}
0=X^{2n}+  g,\\
0=(X^{2}+TZ)^{n}+g_{\h},\\  
0=n(2Z-\overline{\lambda}(\boldsymbol{\mu}\cdot\mathbf{S}))(X^{2}+TZ)^{n-1}
+\frac{\partial g_{\mathbf{h}}}{\partial T},\\
0=
\frac{\partial g_{\mathbf{h}}}{\partial S_{2}}\frac{\partial g}{\partial S_{1}}-\frac{\partial g_{\mathbf{h}}}{\partial S_{1}}\frac{\partial g}{\partial S_{2}}.\end{cases}
$$
Moreover, it is clear that the polynomial $X^{2n}+g$ does not divide the polynomial in the last equation. Hence the 1st and 4th equation meet in a curve in $\PP_{\FF_q}^3$.
On the other hand, the polynomial $(x^{2}+tZ)^{n}+g_{\mathbf{h}}(s_{1},s_{2},t)$ is not constant in $Z$ when $t\neq 0$. 
This implies that any component of $V$ that is not contained in the hyperplane 
$T=0$ has  dimension $1$. Moreover, the intersection of $V$ with the hyperplane 
$T=0$ cuts out the the  system of equations
\[
\begin{cases}
0=X^{2n}+g,\\
0=X^{2n-2}(2Z -\overline{\lambda}(\boldsymbol{\mu}\cdot\mathbf{S})) +\mathbf{h}\cdot\nabla g(S_{1},S_{2})=0,\\
0=T,
\end{cases}
\]
which also has dimension $1$. Hence $\dim (V)=1$, as desired.

We now deal with the case where $f$ does not vanish identically. 
In view of 
\eqref{eq:venice}, 
if we multiply the second equation of 
\eqref{eq:napoli} by $U_1^{n}$ we obtain
\[
0=U_1^{n}\left((X^{2}+TZ)^{n}+  (X^{2}+TZ) f_{\mathbf{h}}+g_{\mathbf{h}}\right)=U_2^{n}+U_2U_1^{n-1}f_{\mathbf{h}}+U_1^{n}g_{\mathbf{h}}.
\]
It follows that the points of $V$ are solutions to the system
$$
\begin{cases}
0=X^{2n}+  X^{2} f(S_{1},S_{2})+g(S_{1},S_{2}),\\
0=U_2^{n}+U_2U_1^{n-1}f_{\mathbf{h}}+U_1^{n}g_{\mathbf{h}},\\
0=n(2Z-\overline{\lambda}(\boldsymbol{\mu}\cdot\mathbf{S}))(X^{2}+TZ)^{n-1}+(2Z-\overline{\lambda}(\boldsymbol{\mu}\cdot\mathbf{S}))f_{\h}
\\
\qquad +(X^{2}+TZ)\frac{\partial f_{\mathbf{h}}}{\partial T}+\frac{\partial g_{\mathbf{h}}}{\partial T},\\
U_2=(X^{2}+TZ)U_{1},
\end{cases}
$$
where $U_1,U_2\in \FF_q[S_1,S_2,T,X]$ are given by \eqref{eq:cornetto}.
We proceed by proving the following fact.

\begin{lem}\label{lem:teramo}
The polynomial $X^{2n}+X^{2}f+g$ does not divide the polynomial $U_2^{n}+U_2U_1^{n-1}f_{\mathbf{h}}+U_1^{n}g_{\mathbf{h}}$.
\end{lem}

\begin{proof}
We can assume that $g$ is separable over $\FF_q$ since $p\nmid \Delta_{f,g}$.  
We argue by  contradiction,  by assuming that 
 $X^{2n}+X^{2}f+g$  divides the polynomial $U_2^{n}+U_2U_1^{n-1}f_{\mathbf{h}}+U_1^{n}g_{\mathbf{h}}$.
Taking  $X=0$, this  implies that the  polynomial $g$  divides
\[
\begin{split}
F&=U_2(S_{1},S_{2},T,0)^{n}+U_2(S_{1},S_{2},T,0)U_1(S_{1},S_{2},T,0)^{n-1}f_{\mathbf{h}}(S_{1},S_{2},T)\\&\quad +U_1(S_{1},S_{2},T,0)^{n}g_{\mathbf{h}}(S_{1},S_{2},T).
\end{split}
\]
Choose  $[s_{1},s_{2}]\in \PP_{\FF_q}^1$ such that 
$g(s_1,s_2)=0$, with  
$\nabla g(s_{1},s_{2})$ not proportional to 
$\nabla g(h_{1},h_{2})$.
This is possible by part (i) of Lemma \ref{lem : polf},  which shows that 
$\nabla g(h_{1},h_{2})$ can be proportional  to at most one of the vectors $\nabla g(s_{1},s_{2})$. It also follows  from   
part (i) of Lemma \ref{lem : polf} that 
 $[h_{1},h_{2}]\neq [s_{1},s_{2}]$.
For this choice of $s_1,s_2$, the polynomial
$U_2(s_{1},s_{2},T,0)$ has degree  $2n-1$, with non-zero leading  coefficient 
\begin{align*}
\frac{\partial g}{\partial S_{2}}(h_{1},h_{2})\frac{\partial g}{\partial S_{1}}(s_{1},s_{2})-\frac{\partial g}{\partial S_{1}}(h_{1},h_{2})\frac{\partial g}{\partial S_{2}}(s_{1},s_{2}).
\end{align*}
Since $g\mid F$ we get
\begin{align*}
0\equiv~&U_1(s_{1},s_{2},T,0)^{n}g_{\mathbf{h}}(s_{1},s_{2},T)+U_2(s_{1},s_{2},T,0)^{n}\\
&\quad +U_2(s_{1},s_{2},T,0)U_1(s_{1},s_{2},T,0)^{n-1}f_{\mathbf{h}}(s_{1},s_{2},T),
\end{align*}
identically in $T$. In particular it follows that  $U_1(s_{1},s_{2},T,0) \mid U_2(s_{1},s_{2},T,0)$, so that  there exists $W\in\overline{\mathbb{F}}_{q}[T]$ such that
\[
U_2(s_{1},s_{2},T,0)=U_1(s_{1},s_{2},T,0)\cdot W(T).
\]
Moreover, since $U_2(s_{1},s_{2},T,0)$ has degree $2n-1$ and $U_1(s_1,s_2,T,0)$ has degree at most $2n-3$, we conclude that  $\deg W\geq 2$. Thus we have 
\[
g_{\mathbf{h}}(s_{1},s_{2},T)=W(T)^n+W(T)\cdot f_{\mathbf{h}}(s_{1},s_{2},T),
\]
identically in $T$. This implies that $W(T)|g_{\mathbf{h}}(s_{1},s_{2},T)$. On the other hand, $g_{\mathbf{h}}(s_{1},s_{2},T)=g(s_{1}+h_{1}T,s_{2}+h_{2}T)$ is a separable polynomial by part 
(ii) of Lemma \ref{lem : polf}, with degree at least $2n-1$. 
Thus $W$ is a separable polynomial of degree $\geq 2$. It  follows that there exists $t\neq 0$ such that
\[
U_2(s_{1},s_{2},t,0)=g(s_{1}+h_{1}t,s_{2}+h_{2}t)=0.
\]
If $s_{1}+h_{1}t=s_{2}+h_{2}t=0$ then $s_{1}h_{2}-s_{2}h_{1}=0$, which implies that $[h_{1},h_{2}]=[s_{1},s_{2}]$  in $\PP_{\FF_q}^1$. This is impossible, by our construction of $[s_{1},s_{2}]$.

We now put  $\tilde{s}_{1}=s_{1}+h_{1}t$ and $\tilde{s}_{2}=s_{2}+h_{2}t$. 
We have already seen that $(\tilde{s}_1,\tilde{s}_2)\neq (0,0)$.
Moreover, by part (ii) of Lemma \ref{lem : polf}, 
$[s_{1},s_{2}]\neq [\tilde{s}_{1},\tilde{s}_{2}]$, since $t\neq 0$.
Hence  there exists $[\tilde{s}_{1},\tilde{s}_{2}]\in \PP_{\FF_q}^1$,  which is a root of $g$ distinct from 
$[s_{1},s_{2}]$, such that
\[
\frac{\partial g}{\partial S_{2}}(\tilde{s}_{1},\tilde{s}_{2})\frac{\partial g}{\partial S_{1}}(s_{1},s_{2})-\frac{\partial g}{\partial S_{1}}(\tilde{s}_{1},\tilde{s}_{2})\frac{\partial g}{\partial S_{2}}(s_{1},s_{2})=0.
\]
This  contradicts part (i) of Lemma \ref{lem : polf}, which thereby completes the proof.
\end{proof}

It follows from Lemma \ref{lem:teramo} that  the system
\[
\begin{cases}
0=X^{2n}+X^{2}f+g,\\
0=U_2^{n}+U_2U_1^{n-1}f_{\mathbf{h}}+U_1^{n}g_{\h},
\end{cases}
\]
defines a variety of dimension $1$ in $\PP_{\FF_q}^{3}$. On the other hand, for any $s_1,s_2,x,t\in \overline\FF_q$, the polynomial
\[
\begin{split}
&n(2Z-\overline{\lambda}(\boldsymbol{\mu}\cdot\mathbf{s}))(x^{2}+tZ)^{2}+(2Z-\overline{\lambda}(\boldsymbol{\mu}\cdot\mathbf{s}))f_{\mathbf{h}}(s_{1},s_{2},t)\\&\quad +(x^{2}+tZ)\frac{\partial f_{\mathbf{h}}}{\partial T}(s_{1},s_{2},t)+\frac{\partial g_{\mathbf{h}}}{\partial T}(s_{1},s_{2},t)
\end{split}
\]
has degree $0$ in $Z$ if  and only if $t=f(s_{1},s_{2})=0$.
It follows that any components of $V$ that are not contained in the intersection of 
$T=0$ with $f(S_1,S_2)=0$ have dimension $1$.  On the other hand, the intersection of $V$ with the variety $T=f(S_1,S_2)=0$ has dimension $1$, since $X$ is constrained by the equation
$X^{2n}+g(S_1,S_2)=0$.
This finally  completes the proof of Proposition \ref{prop : dim1}.

\appendix 

\section{Hooley's method of moments for exponential sums} 

Our work uses a  general procedure due to Hooley
\cite{hooleysum}, which allows one to estimate a very general family of exponential sums over a finite field, provided that one can control the second moment of an appropriate counting function. 

\begin{thm}[Hooley]
Let $F, G_{1},\dots,G_{k}\in \ZZ[X_1,\dots,X_m]$ be polynomials of degree at most $d$ and let
\[
S=\sum_{\substack{x\in\mathbb{F}_{p}^{m}\\G_{1}(\mathbf{x})=\cdots = G_{k}(\mathbf{x})=0}}e_{p}(F(\mathbf{x})).
\]
For each $r\geq 1$ and $\tau\in\mathbb{F}_{p^{r}}$, write
\begin{equation}\label{eq:Nr}
N_{r}(\tau)=\#\left\{\mathbf{x}\in\mathbb{F}_{p^{r}}^{m}: G_{1}(\mathbf{x})=\dots = G_{k}(\mathbf{x})=0, F(\mathbf{x})=\tau\right\}. 
\end{equation}
If there exist $N_{r}\in\mathbb{R}$ such that
\[
\sum_{\tau\in\mathbb{F}_{p^{r}}}|N_{r}(\tau)-N_{r}|^{2}\ll_{d,k,m} p^{\kappa r},
\]
where $\kappa\in\mathbb{Z}$ is independent of $r$, then $S\ll_{d,k,m} p^{\kappa/2}$.
\label{thm:hooley}
\end{thm}

This result relies crucially on Deligne's resolution of the Weil conjectures and can be extracted from 
work of  Hooley \cite{hooleysum}. 
At the recommendation of one of the anonymous referees we will give a full proof of Theorem \ref{thm:hooley} in this appendix. 

Let $\psi:\FF_q\to \CC$ be any non-trivial additive character, 
where $q=p^r$. 
Define 
$$
S_r(\mu)=\sum_{\substack{x\in\mathbb{F}_{q}^{m}\\G_{1}(\mathbf{x})=\cdots = G_{k}(\mathbf{x})=0}}
\psi(\mu F(\mathbf{x})),
$$
for any  $\mu\in \FF_q^*$. We clearly have  $S=S_1(\mu)$, for a suitable $\mu\in \FF_p^*$.
The idea is to study the moment
$$
M_r=\sum_{\mu\in \FF_q^*} |S_r(\mu)|^2.
$$
Recall the definition 
 \eqref{eq:Nr} of  $N_r(\tau)$. Then clearly
$$
S_r(\mu)=
\sum_{\tau\in \FF_q} N_r(\tau) \psi(\mu\tau)=
\sum_{\tau\in \FF_q}\left(N_r(\tau)-N_r\right)\psi(\mu\tau),
$$
for any $N_r\in \RR$. It follows from orthogonality of characters that 
$$
M_r\leq \sum_{\mu\in \FF_q} |S_r(\mu)|^2= q\sum_{\tau\in \FF_q} |N_r(\tau)-N_r|^2.
$$
By hypothesis, there exists $N_r\in\RR$ such that 
\begin{equation}\label{eq:step1}
M_r  \ll_{d,k,m} q^{\kappa +1},
\end{equation}
for some $\kappa\in \ZZ$ that is independent of $r$.

Now fix $\mu\in \FF_p^*$. Associated to $S_r(\mu)$
is a zeta function whose rationality is assured by the work of Dwork \cite{dwork}.
Thus there exists numbers 
$\alpha_1,\dots,\alpha_N\in \CC$ such that 
\begin{equation}\label{eq:dwork}
S_r(\mu)=\sum_{1\leq j\leq n}\alpha_j^r-\sum_{n<j\leq N} \alpha_j^r,
\end{equation}
where $N=O_{d,m,k}(1)$. Furthermore, it follows from Deligne's resolution of the Weil conjectures \cite{weil1} that 
there exist integers $m_1,\dots,m_N\geq 0$ such that 
$$
|\sigma(\alpha_j)|\leq p^{m_j/2}, \quad (1\leq j\leq N)
$$
for any automorphism $\sigma$ of $\bar \QQ$.
We claim that the integers 
 $m_j$ are independent of $\mu$, for $\mu\in \FF_p^*$.  
Let $\beta_1,\dots,\beta_N\in \CC$ be such that 
$S_r(1)=\sum_{1\leq j\leq n}\beta_j^r-\sum_{n<j\leq N} \beta_j^r$ and let 
$\sigma:\bar\QQ\to \bar\QQ$ be the automorphism that takes $\psi(1)$ to $\psi(\mu)$.
Then $S_r(\mu)=\sigma(S_r(1))$, whence $\alpha_j=\sigma(\beta_j)$ for $1\leq j\leq N$.
The claim readily follows.

Let $H=\max_{1\leq j\leq N} m_j$ and let $J=\{j\leq N: m_j=H \}$. We shall prove that 
$J $ is empty when $H\geq \kappa+1$. From this it will follow that 
$|\alpha_j|\leq q^{\kappa}$ for each $1\leq j\leq N$, whence
$$
S_r(\mu)\ll_{d,m,k} q^{\kappa/2},
$$
which is satisfactory for Theorem \ref{thm:hooley}.
Assume that $H\geq \kappa+1$. For each 
 $j\in J$, we write
$\alpha_j=\omega_j p^{H/2}$ for a root of unity $\omega_j$ depending on $\mu$. 
Then it follows from \eqref{eq:dwork} that
$$
S_r(\mu)=q^{H/2} \sum_{j\in J} \omega_j^r+O_{d,m,k}(q^{(H-1)/2}).
$$
But then 
$$
q^{-H} |S_r(\mu)|^2=\left|\sum_{j\in J} \omega_j^r\right|^2 +O_{d,m,k}(q^{-1/2}).
$$
Hence 
$$
q^{-H}M_r\geq q^{-H} \sum_{\mu\in \FF_p^*}|S_r(\mu)|^2 =
\sum_{\mu\in \FF_p^*} \left|\sum_{j\in J} \omega_j^r\right|^2 +O_{d,m,k}(pq^{-1/2}).
$$
Combining this with \eqref{eq:step1}, we deduce that 
$$
\sum_{\mu\in \FF_p^*} \left|\sum_{j\in J} \omega_j^r\right|^2 \ll_{d,m,k} p^{1-r/2} +1.
$$
For each $j\in J$, suppose that $\omega_j=\exp(2\pi i a_j/q_j)$ for appropriate coprime integers $a_j$ and $q_j$. Choose $r\in \ZZ$ such that  $r\geq 2$ and $q_j\mid r$, for $j\in J$. Then it will follow that 
$\omega_j^r=1$, for each $j\in J$, whence
$$
(p-1)\#J^2=O_{d,m,k}(1).
$$
This implies that $J$ is emtpy, as claimed, provided that $p\gg_{d,m,k} 1$.

\end{document}